\documentclass{amsart}

\usepackage[english]{babel}
\usepackage[T1]{fontenc}
\usepackage{ae}
\usepackage{amsthm}
\usepackage{amsfonts}
\usepackage{amsmath}
\usepackage{amssymb}
\usepackage{fancyhdr}
\usepackage{bm}
\usepackage[mathscr]{eucal}
\usepackage[hmarginratio = 1:1]{geometry}
\usepackage{layout}
\usepackage{enumitem}
\usepackage[hidelinks]{hyperref}
\usepackage[utf8]{inputenc}
\usepackage{indentfirst}
\usepackage[dvipsnames]{xcolor}

\makeatletter
\@namedef{subjclassname@2020}{%
  \textup{2020} Mathematics Subject Classification}
\makeatother

\newtheorem{theorem}{Theorem}[section]
\newtheorem{lemma}[theorem]{Lemma}

\newtheorem{proposition}[theorem]{Proposition}

\theoremstyle{definition}
\newtheorem{definition}[theorem]{Definition}
\newtheorem{example}[theorem]{Example}

\theoremstyle{remark}
\newtheorem{remark}[theorem]{Remark}

\numberwithin{equation}{section}

\begin{document}

\title{Differential Norms and Rieffel Algebras}
\author{Rodrigo A. H. M. Cabral}
\address{Departamento de Matem\'atica, Instituto de Matem\'atica e Estat\'istica, Universidade de S\~ao Paulo (IME-USP), BR-05508-090, S\~ao Paulo, SP, Brazil.}
\email{rahmc@ime.usp.br; rodrigoahmc@gmail.com}

\author{Michael Forger}
\address{Departamento de Matem\'atica Aplicada, Instituto de Matem\'atica e Estat\'istica, Universidade de S\~ao Paulo (IME-USP), BR-05508-090, S\~ao Paulo, SP, Brazil.}
\email{forger@ime.usp.br}

\author{Severino T. Melo}
\address{Departamento de Matem\'atica, Instituto de Matem\'atica e Estat\'istica, Universidade de S\~ao Paulo (IME-USP), BR-05508-090, S\~ao Paulo, SP, Brazil.}
\email{toscano@ime.usp.br}


\dedicatory{}
%    "Communicated by" -- provide editor's name; required.
\commby{}
\begin{abstract}
We develop criteria to guarantee uniqueness of the C$^*$-norm on a $*$-algebra $\mathcal{B}$. \linebreak Nontrivial examples are provided by the noncommutative algebras of $\mathcal{C}$-valued functions $\mathcal{S}_J^\mathcal{C}(\mathbb{R}^n)$ and $\mathcal{B}_J^\mathcal{C}(\mathbb{R}^n)$ defined by M.A.~Rieffel via a deformation quantization procedure, where $\mathcal{C}$ is a C$^*$-algebra and $J$ is a skew-symmetric linear transformation on $\mathbb{R}^n$ with respect to which the usual pointwise product is deformed. In the process, we prove that the Fr\'echet $*$-algebra topology of $\mathcal{B}_J^\mathcal{C}(\mathbb{R}^n)$ can be generated by a sequence of submultiplicative $*$-norms and that, if $\mathcal{C}$ is unital, this algebra is closed under the C$^\infty$-functional calculus of its C$^*$-completion. We also show that the algebras $\mathcal{S}_J^\mathcal{C}(\mathbb{R}^n)$ and $\mathcal{B}_J^\mathcal{C}(\mathbb{R}^n)$ are spectrally invariant in their respective C$^*$-completions, when $\mathcal{C}$ is unital. As a corollary of our results, we obtain simple proofs of certain estimates in $\mathcal{B}_J^\mathcal{C}(\mathbb{R}^n)$.
\end{abstract}

\maketitle

\section{Introduction}

The main aim of this paper is to present criteria for a given $*$-algebra, denoted in what follows by $\mathcal{B}$ (and defined purely algebraically as in \cite[p.~35]{murphy}), to admit a unique C$^*$-norm. Let us recall that existence of a C$^*$-norm already imposes restrictions, since there are examples of $*$-algebras which do not admit any C$^*$-norm at all. And even when C$^*$-norms do exist, there may ``a priori'' be many different ones -- see the beginning of Section \ref{uniqtheorems}. It is true that any two C$^*$-norms on a $*$-algebra turning it into a C$^*$-algebra are necessarily equal \cite[Corollary 2.1.2, p.~37]{murphy}, but the conclusion breaks down when we abandon the hypothesis of completeness.

Since our focus here will be on the question of uniqueness and not of existence, we shall in the sequel bypass the latter by assuming that the $*$-algebra $\mathcal{B}$ in question is realized as a dense $*$-subalgebra of some C$^*$-algebra $\mathcal{A}$. Within this context we formulate our first main theorem, which states that if $\mathcal{B}$ is closed under the C$^\infty$-functional calculus of $\mathcal{A}$ (see Definition \ref{onlysmooth}), then the C$^*$-norm on $\mathcal{B}$ induced from that of $\mathcal{A}$ is the only possible one (Theorem \ref{uniq}). In the unital case, this can be seen as a noncommutative version of the statement that on a smooth compact manifold $M$, the algebra $C^\infty(M)$ of smooth functions uniquely determines the algebra $C(M)$ of continuous functions, which is the algebraic counterpart of the idea that a smooth manifold is automatically also a topological space: the smooth structure uniquely determines the topology \cite[Chapter 2, pp.~22 \& 23]{lang}.

In our main applications, $\mathcal{B}$ will not be merely a $*$-algebra but rather a Fr\'echet $*$-algebra, that is, a $*$-algebra which is also a metrizable and complete locally convex topological vector space such that both its multiplication and its involution are continuous.\footnote{We recall that a separately continuous bilinear map from a Fr\'echet space to an arbitrary locally convex space is automatically (jointly) continuous \cite[Theorem 1, p.~357]{horvath} \cite[p.~214]{horvath}.} A particularly interesting situation appears when $\mathcal{B}$ is a Fr\'echet $*$-algebra whose topology can be defined by a differential seminorm, as originally introduced by B.~Blackadar and J.~Cuntz \cite{blackadarcuntz} and later modified by S.J.~Bhatt, A.~Inoue and H.~Ogi \cite[Definition~3.1]{bhattdiff}. In this case, there are important results \cite[Theorems 3.3 \& 3.4]{bhattdiff} which will guarantee the validity of the hypotheses of Proposition \ref{envalg}, Theorem \ref{uniq2} and Theorem \ref{uniq}.

In Section \ref{rieffel}, we consider the noncommutative function algebras $\mathcal{S}_J^\mathcal{C}(\mathbb{R}^n)$ and $\mathcal{B}_J^\mathcal{C}(\mathbb{R}^n)$ defined by M.A.~Rieffel \cite{rieffel} via a deformation quantization procedure, where $\mathcal{C}$ is a given C$^*$-algebra of ``coefficients'' and $J$ is a skew-symmetric linear transformation on $\mathbb{R}^n$ with respect to which the usual (commutative) pointwise product is ``deformed'' (Definition \ref{rieffelalgebras}). Using Rieffel's deformed product $\times_J$ (see Equation \eqref{rieffelprod}), together with the pointwise involution and with the choice of standard systems of (semi)norms which are familiar from the theory of distributions (there are several variants), these are Fr\'echet $*$-algebras. However, we will substitute the initial system of norms on $\mathcal{B}_J^\mathcal{C}(\mathbb{R}^n)$ by a more convenient one, resorting to a faithful representation of this function algebra as an algebra of bounded (pseudodifferential) operators on a Hilbert C$^*$-module \cite{lance}. More precisely, we first define an ``operator C$^*$-norm'' on $\mathcal{B}_J^\mathcal{C}(\mathbb{R}^n)$ (see Definition \ref{rieffelalgebras}) and, under the assumption of a unital $\mathcal{C}$, we will define a differential norm on $\mathcal{B}_J^\mathcal{C}(\mathbb{R}^n)$, a construction which will require several steps. In particular, we will need a version of the Calder\'on-Vaillancourt inequality for Hilbert C$^*$-modules (see Theorem \ref{calderonvaillancourt} and Equation \eqref{calderon}), as well as the ``symbol map'' $S$ constructed in reference \cite{melomerklen2} that allows us to obtain an ``inverse Calder\'on-Vaillancourt-type inequality'' (see Equation \eqref{invcalderon}) which, in the scalar case ($\mathcal{C} = \mathbb{C}$), was proved by H.O.~Cordes in \cite[Proposition 4.2, p.~262]{cordes}. Besides showing that the natural topology of $\mathcal{B}_J^\mathcal{C}(\mathbb{R}^n)$ is, in particular, defined by a sequence of submultiplicative $*$-norms, the fact that the topology on $\mathcal{B}_J^\mathcal{C}(\mathbb{R}^n)$ is generated by a differential norm (Theorem \ref{b-diffnorm}) also implies, for a unital $\mathcal{C}$, that this $*$-algebra is closed under the C$^\infty$-functional calculus of its C$^*$-completion (see Theorem \ref{b-uniq}). This result will put us in a position to establish the uniqueness statement for C$^*$-norms on $\mathcal{B}_J^\mathcal{C}(\mathbb{R}^n)$, by means of Theorem \ref{uniq}, for any C$^*$-algebra $\mathcal{C}$ (unital, or not -- see Theorem \ref{b-uniq2}). The analogous C$^*$-norm uniqueness statement for $\mathcal{S}_J^\mathcal{C}(\mathbb{R}^n)$ will also be obtained as a corollary, in Theorem \ref{s-uniq}. Moreover, due to the spectral invariance results contained in Theorems \ref{b-uniq} and \ref{s-specinv} (for a unital $\mathcal{C}$), $\mathcal{B}_J^\mathcal{C}(\mathbb{R}^n)$ and $\mathcal{S}_J^\mathcal{C}(\mathbb{R}^n)$ have the same $K$-theory as their respective C$^*$-completions.

At the end of Section \ref{rieffel}, we provide a few other applications. We begin by showing that the Fr\'echet $*$-algebra of smooth elements for a strongly continuous Lie group representation by $*$-automorphisms on a C$^*$-algebra admits only one C$^*$-norm (Theorem \ref{smooth}), illustrating this result with two algebras of pseudodifferential operators with scalar-valued symbols. Then, we prove that the ``sup norm'' and the ``operator C$^*$-norm'' coincide on $\mathcal{B}_J^\mathcal{C}(\mathbb{R}^n)$ when $J = 0$ (Proposition \ref{sup-op}). Finally, in Theorem \ref{applications}, we use some of our results to give very simple proofs of three propositions of Rieffel's monograph \cite{rieffel}: Propositions 4.11, 5.4 and 5.6.

\section{Uniqueness of C\texorpdfstring{$^*$}{*}-norms} \label{uniqtheorems}

As observed in the Introduction, some $*$-algebras may not admit any C$^*$-norm at all. For a concrete example, denote the Schwartz function space by $\mathcal{S}(\mathbb{R}^n)$ (see Section \ref{rieffel}), which is a dense subspace of $L^2(\mathbb{R}^n)$. Also, consider the algebra $\text{End}^+(\mathcal{S}(\mathbb{R}^n))$ of all linear operators $T$ on $L^2(\mathbb{R}^n)$ such that $\text{Dom }T := \mathcal{S}(\mathbb{R}^n)$, $T[\mathcal{S}(\mathbb{R}^n)] \subseteq \mathcal{S}(\mathbb{R}^n)$, $\mathcal{S}(\mathbb{R}^n) \subseteq \text{Dom }T^*$ and $T^*[\mathcal{S}(\mathbb{R}^n)] \subseteq \mathcal{S}(\mathbb{R}^n)$, where $T^*$ denotes the adjoint operator on $L^2(\mathbb{R}^n)$. Then $\text{End}^+(\mathcal{S}(\mathbb{R}^n))$ becomes a $*$-algebra when equipped with the involution operation $T \longmapsto T^+ := T^*|_{\mathcal{S}(\mathbb{R}^n)}$ \cite[Lemma 3.2, p.~40]{schmudgen2020}. Moreover, define $\mathcal{B}$ as the $*$-subalgebra of $\text{End}^+(\mathcal{S}(\mathbb{R}^n))$ generated by the set
\begin{equation*}
\left\{a_k := \frac{\partial}{\partial x_k} \Big|_{\mathcal{S}(\mathbb{R}^n)}, \, b_k := x_k|_{\mathcal{S}(\mathbb{R}^n)}, \, I|_{\mathcal{S}(\mathbb{R}^n)}: 1 \leq k \leq n\right\}
\end{equation*}
of linear operators on $\mathcal{S}(\mathbb{R}^n)$, in which $x_k$ is the multiplication operator by the coordinate function $x \longmapsto x_k$ and $I$ is the identity operator on $L^2(\mathbb{R}^n)$. Then since
\begin{equation*}
(a_k \circ b_k - b_k \circ a_k)(f) = \frac{\partial (x_k(f))}{\partial x_k} - x_k \, \frac{\partial f}{\partial x_k} = f, \qquad f \in \mathcal{S}(\mathbb{R}^n), \, 1 \leq k \leq n,
\end{equation*}
we conclude, as a consequence of the fact that in a unital Banach algebra a commutator of two elements cannot be equal to the identity \cite[Theorem 13.6, p.~351]{rudinfunctionalanalysis}, that neither $\text{End}^+(\mathcal{S}(\mathbb{R}^n))$ nor $\mathcal{B}$ can carry any submultiplicative norm, let alone a C$^*$-norm.
 
On the other hand, some $*$-algebras can be equipped with more than one C$^*$-norm (in fact, with many different ones). For example, if $S^1 := \left\{ z \in \mathbb{C}: |z| = 1 \right\}$ and $C(S^1)$ is the $*$-algebra of complex-valued continuous functions on $S^1$, then the $*$-subalgebra
\begin{equation*}
\mathcal{B} := \left\{p \colon S^1 \ni z \longmapsto p(z) = \sum_{k = - n}^n a_k \, z^k, a_k \in \mathbb{C}, n \in \mathbb{N} \right\}
\end{equation*}
of trigonometric polynomials admits infinitely many C$^*$-norms: the C$^*$-seminorms $\|\, \cdot \,\|_K \colon g \longmapsto \sup \left\{ |g(z)|: z \in K \right\}$ on $C(S^1)$, in which $K$ is an infinite compact subset of $S^1$, restrict to C$^*$-norms on $\mathcal{B}$, as a consequence of the Identity Theorem for holomorphic functions \cite[The Identity Theorem, p.~228]{remmert1}. To see that $\|\, \cdot \,\|_{K_1} \neq \|\, \cdot \,\|_{K_2}$ if $K_1$ and $K_2$ are distinct infinite compact subsets of $S^1$ first note that, if $z_0 \in K_1 \backslash K_2$, then there exists a compactly supported continuous function $0 \leqslant f \leqslant 1$ on $S^1$ such that $f(z_0) = 1$ and $f|_{K_2} = 0$, so $\|f\|_{K_2} = 0$ and $\|f\|_{K_1} \geqslant 1$; therefore, since the trigonometric polynomials form a dense subalgebra of $C(S^1)$ with respect to the C$^*$-norm $\|\, \cdot \,\|_{S^1} \colon g \longmapsto \sup \left\{ |g(z)|: z \in S^1 \right\}$ \cite[Theorem 4.25, p.~91]{rudinrealandcomplex}, there must be an element $p_0$ in $\mathcal{B}$ such that $\|p_0\|_{K_1} \neq \|p_0\|_{K_2}$.

In the discussion of the two examples above, we have already followed what we believe to be standard terminology in the literature, according to which a seminorm $p$ on a $*$-algebra $\mathcal{B}$ is just a seminorm on $\mathcal{B}$ as a vector space, so the term ``seminorm'' in itself does not ``a priori'' include any requirement of compatibility with either the multiplication or the involution on $\mathcal{B}$. Correspondingly, we say that a seminorm $p$ is \textit{submultiplicative} if we have $p(b_1 b_2) \leqslant p(b_1) p(b_2)$, for all $b_1, b_2 \in \mathcal{B}$, is a \textit{$*$-seminorm} if $p(b^*) = p(b)$, for all $b \in \mathcal{B}$ and is a \textit{C$^*$-seminorm} if it is a submultiplicative $*$-seminorm satisfying $p(b^*b) = p(b)^2$, for all $b \in \mathcal{B}$. Finally, throughout the paper we shall often employ the notation $\|\, \cdot \,\|_\mathcal{B}$ to denote a C$^*$-norm on a general $*$-algebra $\mathcal{B}$. When $\mathcal{A}$ is a C$^*$-algebra, for example, $\|\, \cdot \,\|_\mathcal{A}$ will denote the unique C$^*$-norm which may be defined on $\mathcal{A}$.

The uniqueness theorem for C$^*$-norms on certain $*$-algebras that we shall prove in this section (Theorem \ref{uniq}) depends on just one essential condition, namely closure under the C$^\infty$ functional calculus. But in an intermediate step (Theorem \ref{uniq2}), it involves two technical conditions, one of which is a weakened form of spectral invariance (see Definition \ref{specinv} below).

We begin by recalling the definition of \textit{spectrum of an element} $a$ of an algebra $\mathcal{A}$: if $\mathcal{A}$ is unital with unit $1_\mathcal{A}$, it is the set $\sigma_\mathcal{A}(a) \subseteq \mathbb{C}$ of numbers $\lambda$ such that $\lambda 1_\mathcal{A} - a$ is not invertible in $\mathcal{A}$, whereas if $\mathcal{A}$ is non-unital, it is defined to be the spectrum of $(a, 0)$ in the unitization $\tilde{\mathcal{A}}$ of $\mathcal{A}$ \cite[pp.~6 \& 12]{murphy}. If $\sigma_\mathcal{A}(a) \neq \emptyset$, the corresponding \textit{spectral radius} of $a$ is defined to be $r_\mathcal{A}(a) := \sup \left\{ |\lambda|: \lambda \in \sigma_\mathcal{A}(a) \right\}$.

\begin{proposition} \label{envalg}
Let $\mathcal{B}$ be a dense $*$-subalgebra of a C$^*$-algebra $\mathcal{A}$ with the property that $r_\mathcal{B}(b^*b) = r_\mathcal{A}(b^*b)$, for all $b \in \mathcal{B}$. Then every C$^*$-seminorm on $\mathcal{B}$ is majorized by the restriction of $\|\, \cdot \,\|_\mathcal{A}$ to $\mathcal{B}$. \end{proposition}

\begin{proof} For each $*$-representation $\rho$ of $\mathcal{B}$ on a Hilbert space $\mathcal{H}$, we may use the corresponding operator norm $\|\, \cdot \,\|_{\mathcal{L}(\mathcal{H})}$ to define a C$^*$-seminorm $\|\, \cdot \,\|_\rho$ on $\mathcal{B}$ by $\|b\|_\rho := \|\rho(b)\|_{\mathcal{L}(\mathcal{H})}$, for all $b \in \mathcal{B}$. Moreover, every C$^*$-seminorm $p$ on $\mathcal{B}$ is of this form: the completion $\overline{\mathcal{B}/\text{ker }p}$ of the quotient of $\mathcal{B}$ by the kernel of $p$ (with respect to the C$^*$-norm given by $\|[b]\|_p := p(b)$, for every $[b] \in \mathcal{B}/\text{ker }p$) is a C$^*$-algebra which, according to the Gelfand-Naimark Theorem \cite[Theorem 3.4.1, p.~94]{murphy} has a faithful representation $\rho'$ on some Hilbert space; therefore, composition of $\rho'$ with the canonical projection from $\mathcal{B}$ to $\mathcal{B}/\text{ker }p$ produces a $*$-representation of $\mathcal{B}$ whose operator C$^*$-seminorm is equal to $p$.

Now let $p$ be a C$^*$-seminorm on $\mathcal{B}$ and $\rho$ be a $*$-representation of $\mathcal{B}$ on some Hilbert space $\mathcal{H}$ satisfying $p = \|\, \cdot \,\|_\rho$. Then by the hypothesis, for every fixed $b \in \mathcal{B}$,
\begin{equation*}
p(b)^2 = \|b\|_\rho^2 = \|\rho(b)\|_{\mathcal{L}(\mathcal{H})}^2 = \|\rho(b)^* \rho(b)\|_{\mathcal{L}(\mathcal{H})} = \|\rho(b^* b)\|_{\mathcal{L}(\mathcal{H})} = r_{\mathcal{L}(\mathcal{H})}(\rho(b^*b))
\end{equation*}
\begin{equation*}
\leqslant r_\mathcal{B}(b^*b) = r_\mathcal{A}(b^*b) = \|b^*b\|_\mathcal{A} = \|b\|_\mathcal{A}^2.
\end{equation*}
\end{proof}

Similarly as in reference \cite{schweitzer}, we shall adopt the following conventions: let $\mathcal{A}$ be an algebra and $\mathcal{B}$ be a subalgebra of $\mathcal{A}$. If $\mathcal{A}$ is non-unital, we define $\dot{\mathcal{A}}$ as $\tilde{\mathcal{A}}$ and $\dot{\mathcal{B}}$ as $\tilde{\mathcal{B}}$; if $\mathcal{A}$ and $\mathcal{B}$ are both unital and share the same unit, let $\dot{\mathcal{A}} := \mathcal{A}$ and $\dot{\mathcal{B}} := \mathcal{B}$; finally, if $\mathcal{A}$ is unital but the unit $1_\mathcal{A}$ of $\mathcal{A}$ does not belong to $\mathcal{B}$, let $\dot{\mathcal{A}} := \mathcal{A}$ and $\dot{\mathcal{B}}$ be the subalgebra of $\mathcal{A}$ generated by $\mathcal{B}$ and $1_\mathcal{A}$. In any case, $\dot{\mathcal{A}}$ and $\dot{\mathcal{B}}$ are unital algebras sharing the same unit. We now make the following

\begin{definition} \label{specinv} Let $\mathcal{A}$ be an algebra and $\mathcal{B}$ be a subalgebra of $\mathcal{A}$. We say that $\mathcal{B}$ is \textit{spectrally invariant} in $\mathcal{A}$ if, for every element of $\dot{\mathcal{B}}$, its spectrum as an element of $\dot{\mathcal{B}}$ coincides with its spectrum as an element of $\dot{\mathcal{A}}$. Similarly, we say that $\mathcal{B}$ is \textit{real spectrally invariant} in $\mathcal{A}$ (respectively, \textit{positive spectrally invariant} in $\mathcal{A}$) if, for every element $b$ of $\dot{\mathcal{B}}$ satisfying $b = b^*$ (respectively, for every element $b$ of $\dot{\mathcal{B}}$ satisfying $b = c^*c$, for some $c \in \dot{\mathcal{B}}$), its spectrum as an element of $\dot{\mathcal{B}}$ coincides with its spectrum as an element of $\dot{\mathcal{A}}$. \end{definition}

Therefore, a sufficient condition for guaranteeing the hypothesis of ``spectral radius invariance'' for elements of the form $b^*b$, $b \in \mathcal{B}$, in Proposition \ref{envalg}, is obtained by requiring $\mathcal{B}$ to be positive spectrally invariant in $\mathcal{A}$, since in this case we have
\begin{equation*}
r_\mathcal{B}(b^*b) = r_{\dot{\mathcal{B}}}(b^*b) = r_{\dot{\mathcal{A}}}(b^*b) = r_\mathcal{A}(b^*b).
\end{equation*}
Clearly, such a condition is also satisfied if the stronger hypothesis that $\mathcal{B}$ is spectrally invariant in $\mathcal{A}$ is fulfilled. A brief discussion on the issue of spectral invariance may be found in Appendix \ref{appendixa}. As is well known, this condition is equivalent to requiring that whenever an element of $\dot{\mathcal{B}}$ has an inverse in $\dot{\mathcal{A}}$, this inverse already belongs to $\dot{\mathcal{B}}$. We should also mention that concepts very similar to the properties of real spectral invariance and of spectral radius invariance for elements of the form $b^*b$, $b \in \mathcal{B}$, presented above, have been discussed in the literature before; compare, for instance, with the concept of \textit{$*$-inverse closedness} and with the \textit{spectral radius preserving (SRP)} property in \cite{barnes1} and \cite{barnes2}.

We will now see that if $\mathcal{B}$ bears a nice relationship with the closed ideals of $\mathcal{A}$ and satisfies the hypotheses of Proposition \ref{envalg}, then it admits only one C$^*$-norm.

\begin{theorem} \label{uniq2} Let $\mathcal{B}$ be a dense $*$-subalgebra of a C$^*$-algebra $\mathcal{A}$ satisfying the following two hypotheses:
\begin{enumerate}
\item For all $b \in \mathcal{B}$, the equality $r_\mathcal{B}(b^*b) = r_\mathcal{A}(b^*b)$ of spectral radii holds.
\item For every closed ideal $\mathcal{I}$ of $\mathcal{A}$, $\mathcal{I} \cap \mathcal{B}$ is a dense $*$-subalgebra of $\mathcal{I}$.
\end{enumerate}
Then the only C$^*$-norm which may possibly be defined on $\mathcal{B}$ is the restriction of $\|\, \cdot \,\|_\mathcal{A}$. \end{theorem}

\begin{proof} According to Proposition \ref{envalg}, any C$^*$-norm $\|\, \cdot \,\|_\mathcal{B}$ on $\mathcal{B}$ is majorized by the restriction of $\|\, \cdot \,\|_\mathcal{A}$ to $\mathcal{B}$. Therefore, one can extend $\|\, \cdot \,\|_\mathcal{B}$ uniquely to a C$^*$-seminorm $p_\mathcal{B}$ on $\mathcal{A}$, whose kernel $\mathcal{I}$ will be a closed $*$-ideal of $\mathcal{A}$; moreover, due to the fact that $\|\, \cdot \,\|_\mathcal{B}$ is a norm on $\mathcal{B}$, we have $\mathcal{I} \cap \mathcal{B} = \left\{0\right\}$.
By the hypothesis (2), it follows that $\left\{0\right\}$ is dense in $\mathcal{I}$, so $\mathcal{I} = \left\{0\right\}$: in other words, $p_\mathcal{B}$ is actually a C$^*$-norm on $\mathcal{A}$.
Since there exists only one C$^*$-norm turning $\mathcal{A}$ into a C$^*$-algebra, $p_\mathcal{B}$ must coincide with $\|\, \cdot \,\|_\mathcal{A}$ on $\mathcal{A}$ and, in particular, on $\mathcal{B}$. This proves the claim. \end{proof}

Our next objective will be to search for situations in which the requirements (1) and (2) of Theorem \ref{uniq2} are fulfilled in a natural way. At this point we find it appropriate to say a few words about the concept of closure under the C$^\infty$-functional calculus, taking into account the possibility that the larger algebra and its subalgebra may not share a unit:

\begin{definition} \label{onlysmooth} Let $\mathcal{B}$ be a $*$-subalgebra of a C$^*$-algebra $\mathcal{A}$. $\mathcal{B}$ is said to be \textit{closed under the C$^\infty$-functional calculus \cite[p.~309]{bhattdiff} \cite[p.~256, (1)]{blackadarcuntz} \cite[Remark (1), p.~274]{kissin} \cite[p.~22]{longo} (or smooth functional calculus \cite[p.~6]{elliott}) of $\mathcal{A}$} if, for every self-adjoint element $b$ of $\dot{\mathcal{B}}$ and every smooth function $f$ on an open neighborhood $U \subseteq \mathbb{R}$ of $\sigma_{\dot{\mathcal{A}}}(b)$, one has $f(b) \in \dot{\mathcal{B}}$. \end{definition}

The following theorem shows that being closed under the C$^\infty$-functional calculus of $\mathcal{A}$ is a sufficient hypothesis on the dense $*$-subalgebra $\mathcal{B}$ in order to guarantee uniqueness of the C$^*$-norm. Part of its proof adapts an argument which may be found in \cite[Proposition 6.7(b)]{blackadarcuntz} (see also \cite[Lemma 2]{batty}):

\begin{theorem} \label{uniq} Let $\mathcal{B}$ be a dense $*$-subalgebra of a C$^*$-algebra $\mathcal{A}$, closed under the C$^\infty$-functional calculus of $\mathcal{A}$. Then the only C$^*$-norm which may possibly be defined on $\mathcal{B}$ is the restriction of $\|\, \cdot \,\|_\mathcal{A}$.
\end{theorem}

\begin{proof} Let us show that the hypotheses (1) and (2) of Theorem \ref{uniq2} are verified, beginning with (1). Let us prove that, for all $b \in \mathcal{B}$, the equality $\sigma_{\dot{\mathcal{B}}}(b^*b) = \sigma_{\dot{\mathcal{A}}}(b^*b)$ of spectra holds. Fix $b \in \mathcal{B}$ and $\lambda \in \mathbb{C} \backslash \sigma_{\dot{\mathcal{A}}}(b^*b)$. Then by compactness of the spectrum $\sigma_{\dot{\mathcal{A}}}(b^*b)$ there must be an open set $V \subseteq \mathbb{C}$ such that $\lambda \in \overline{V} \subseteq \mathbb{C} \backslash \sigma_{\dot{\mathcal{A}}}(b^*b)$. Therefore, the function $f \colon \mu \longmapsto (\lambda - \mu)^{-1}$ is well-defined and smooth on the open subset $U := \mathbb{R} \cap (\mathbb{C} \backslash \overline{V})$ of $\mathbb{R}$, which contains $\sigma_{\dot{\mathcal{A}}}(b^*b)$. Hence, $(\lambda 1_{\dot{\mathcal{A}}} - b^*b)^{-1} = f(b^*b) \in \dot{\mathcal{B}}$. This proves the inclusion $\sigma_{\dot{\mathcal{B}}}(b^*b) \subseteq \sigma_{\dot{\mathcal{A}}}(b^*b)$ and, since the reverse inclusion is automatic, we have proved the desired statement. Since the equality $\sigma_{\dot{\mathcal{B}}}(b^*b) = \sigma_{\dot{\mathcal{A}}}(b^*b)$ of spectra trivially implies the equality $r_\mathcal{B}(b^*b) = r_\mathcal{A}(b^*b)$ of spectral radii, we have shown that (1) holds.

To prove (2), we first assume that $\mathcal{B}$ and $\mathcal{A}$ are unital algebras sharing the same unit. To show that for every $x \in \mathcal{I}$ and every $\epsilon > 0$, there exists $z \in \mathcal{I} \cap \mathcal{B}$ such that $\|x - z\|_\mathcal{A} < \epsilon$, we may assume without loss of generality that $x^* = x$ (otherwise, apply the following argument to $(x + x^*)/2$ and $(x - x^*)/(2i)$, using that $\mathcal{I}$ is $*$-invariant \cite[Theorem 3.1.3, p.~79]{murphy}). Thus fix an element $x = x^*$ in $\mathcal{I}$ and $\epsilon > 0$. By the denseness hypothesis there exists an element $y$ in $\mathcal{B}$, which once again without loss of generality may be assumed to be self-adjoint, such that $\|x - y\|_\mathcal{A} < \epsilon/3$. Now let $0 \leqslant \chi \leqslant 1$ be a smooth function on $\mathbb{R}$ with support contained in the interval $[-2 \epsilon/3, 2 \epsilon/3]$ such that $\chi(t) = 1$, for all $t \in [-\epsilon/3, \epsilon/3]$. Then the function $f$ defined by $f(t) := t(1 - \chi(t))$ satisfies $\sup_{t \in \mathbb{R}} |f(t) - t| \leqslant 2 \epsilon/3$, so the continuous functional calculus of $\mathcal{A}$ implies that $\|f(y) - y\|_\mathcal{A} \leqslant 2 \epsilon / 3$. Therefore, since $\mathcal{B}$ is closed under the C$^\infty$-functional calculus of $\mathcal{A}$, $f(y)$ is a self-adjoint element of $\mathcal{B}$ such that $\|f(y) - x\|_\mathcal{A} < \epsilon$. To see that $f(y)$ also belongs to $\mathcal{I}$, note that if $\pi \colon \mathcal{A} \longrightarrow \mathcal{A} / \mathcal{I}$ is the canonical quotient map, then $\|\pi(y)\|_{\mathcal{A} / \mathcal{I}} = \|\pi(y - x)\|_{\mathcal{A} / \mathcal{I}} \leqslant \|y - x\|_\mathcal{A} < \epsilon / 3$. This shows, in particular, that $\sigma_{\mathcal{A} / \mathcal{I}}(\pi(y)) \subseteq [-\epsilon/3, \epsilon/3]$, so we conclude that $\pi(f(y)) = f(\pi(y)) = 0$, since $f$ vanishes on $\sigma_{\mathcal{A} / \mathcal{I}}(\pi(y))$. This proves that $f(y)$ belongs to $\mathcal{I} \cap \mathcal{B}$, establishing the density claim. 

Now, we deal with the general case. By what we have proved, $\mathcal{I} \cap \tilde{\mathcal{B}}$ is dense in $\mathcal{I}$, since every closed ideal in $\mathcal{A}$ is a closed ideal in $\tilde{\mathcal{A}}$ (here we are making the usual identification of $\mathcal{I}$ with its image in $\tilde{\mathcal{A}}$ via the canonical inclusion $\mathcal{A} \hookrightarrow \tilde{\mathcal{A}}$). Fix $x \in \mathcal{I}$ and let $((x_n, \lambda_n))_{n \in \mathbb{N}}$ be a sequence in $\mathcal{I} \cap \tilde{\mathcal{B}}$ converging to $x$, where $x_n \in \mathcal{B}$ and $\lambda_n \in \mathbb{C}$, for all $n \in \mathbb{N}$. To establish that $\mathcal{I} \cap \mathcal{B}$ is dense in $\mathcal{I}$ we shall prove that $(x_n)_{n \in \mathbb{N}}$ also converges to $x$. By the definition of the C$^*$-norm of $\tilde{\mathcal{A}}$, it follows that $(\lambda_n)_{n \in \mathbb{N}}$ is a Cauchy sequence, so it converges to a certain $\lambda \in \mathbb{C}$. This implies that $(x_n)_{n \in \mathbb{N}}$ converges to some $y \in \mathcal{A}$, from which it follows that $(x, 0) = (y, \lambda)$. Consequently, $x = y$ and $\lambda = 0$, which proves the desired claim. 

Therefore, uniqueness of the C$^*$-norm on $\mathcal{B}$ is a consequence of Theorem \ref{uniq2}.
\end{proof}

\begin{remark}
We note that, if we substitute $b^*b \in \mathcal{B}$ by a self-adjoint element $b = b^* \in \dot{\mathcal{B}}$ in the first paragraph of the proof of Theorem \ref{uniq} we can establish, with an easy adaptation of the arguments, the following fact: \textit{if $\mathcal{B}$ is a dense $*$-subalgebra of a C$^*$-algebra $\mathcal{A}$, closed under the C$^\infty$-functional calculus of $\mathcal{A}$, then $\mathcal{B}$ is \textit{real spectrally invariant} in $\mathcal{A}$.}
\end{remark}

\begin{remark}
Before carrying on, we would like to point out that, although the hypothesis of being closed under the C$^\infty$-functional calculus is sufficient to guarantee uniqueness of the C$^*$-norm, it is by no means necessary. In fact, let $\mathcal{A} = C(\mathbb{T})$ be the C$^*$-algebra of $(2 \pi)$-periodic complex-valued continuous functions on $\mathbb{R}$, equipped with the C$^*$-norm
\begin{equation*}
\|\, \cdot \,\|_\infty \colon f \longmapsto \|f\|_\infty := \sup_{t \in [- \pi, \pi]} |f(t)|,
\end{equation*}
and $\mathcal{B} = A(\mathbb{T})$ be the $*$-subalgebra of $C(\mathbb{T})$ consisting of functions having an absolutely convergent Fourier series. Then $A(\mathbb{T})$ is dense in $C(\mathbb{T})$, because it contains the trigonometric polynomials, which form a dense $*$-subalgebra of $C(\mathbb{T})$. Moreover, $A(\mathbb{T})$ is, according to the terminology in \cite[Definition (3), p.~269]{kissin}, locally normal in $C(\mathbb{T})$ \cite[Remark (1), p.~275]{kissin}, so \cite[Theorem 13(i), p.~274]{kissin} shows that the hypothesis (2) in Theorem \ref{uniq2} is satisfied. On the other hand, as a consequence of Wiener's theorem \cite[Theorem 5.51, p.~140]{maccluer}, $A(\mathbb{T})$ is spectrally invariant in $C(\mathbb{T})$, which immediately implies hypothesis (1) of Theorem \ref{uniq2}. Therefore, there exists only one C$^*$-norm on $A(\mathbb{T})$, which is obtained by restricting $\|\, \cdot \,\|_\infty$ to $A(\mathbb{T})$. However, as noted in \cite[Remark (1), p.~275]{kissin}, $A(\mathbb{T})$ is not closed under the C$^\infty$-functional calculus of $C(\mathbb{T})$ (see \cite[pp.~80--82]{kahane}, as well as \cite{katznelsoncomptes} and \cite{rudinstrongconverse}). This observation shows that the converse of Theorem \ref{uniq} does not hold, in general.
\end{remark}

Next, we would like to make a few comments about the families of seminorms we shall employ to define the topologies of our Fr\'echet $*$-algebras. The topology of every Fr\'echet $*$-algebra $\mathcal{B}$ can be generated by an increasing sequence of $*$-seminorms $(p_m)_{m \in \mathbb{N}}$ \cite[Theorem 3.7, p.~32]{fragoulopoulou}, meaning that $p_{m_1}(b) \leqslant p_{m_2}(b)$, for all $b \in \mathcal{B}$, whenever $m_1, m_2 \in \mathbb{N}$ satisfy $m_1 \leqslant m_2$. Sometimes, such a topology can even be generated by a family of submultiplicative $*$-seminorms, but not all Fr\'echet $*$-algebras have this property: those that do are often called \textit{Arens-Michael $*$-algebras} \cite[Definition 3.5, p.~30]{fragoulopoulou} (see also the paragraph right before \cite[Proposition 2.3]{michael}). Indeed, the continuity assumption on the multiplication of a Fr\'echet $*$-algebra $\mathcal{B}$ whose topology is generated by an increasing sequence $(p_m)_{m \in \mathbb{N}}$ of $*$-seminorms does not in itself force these to be submultiplicative; rather, it only means that, for each $m \in \mathbb{N}$, there exist $C_m > 0$ and $m' \in \mathbb{N}$ such that
\begin{equation}
 p_m(b_1 b_2) \leqslant C_m \, p_{m + m'}(b_1) \, p_{m + m'}(b_2), \qquad \text{for all } b_1, b_2 \in \mathcal{B},
\end{equation}
and in order for the $*$-seminorm $p_m$ to be submultiplicative, this property would have to hold with $C_m = 1$ and $m' = 0$.

Now, we introduce a central notion for the investigations of the present paper:

\begin{definition} \label{diffnorm} Let $\mathcal{B}$ be a unital C$^*$-normed algebra -- in other words, $\mathcal{B}$ is a (not necessarily complete) unital $*$-algebra whose topology is generated by the C$^*$-norm $\|\, \cdot \,\|_\mathcal{B}$. According to \cite[Definition 3.1]{bhattdiff}, a \textit{differential seminorm} on $\mathcal{B}$ is a map $T \colon b \longmapsto (T_k(b))_{k \in \mathbb{N}}$ on $\mathcal{B}$ assuming values in sequences of non-negative real numbers such that: (1) each $T_k$ is a $*$-seminorm; (2) $T_0(b) \leqslant c \|b\|_\mathcal{B}$, for some $c > 0$ and all $b \in \mathcal{B}$; (3) we have \begin{equation*} T_k(a b) \leqslant \sum_{i + j = k} T_i(a) T_j(b), \qquad a, b \in \mathcal{B},\end{equation*} for all $k \in \mathbb{N}$ -- note that this forces the first seminorm, $T_0$, to be submultiplicative. If $T(b) = 0$ implies $b = 0$, then $T$ is said to be a \textit{differential norm}. \end{definition}

In the examples of interest to us, $T_0$ will always be equal to $\|\, \cdot \,\|_\mathcal{B}$, so $T$ will be a differential norm and the underlying topology generated by the sequence $(T_k)_{k \in \mathbb{N}}$ of $*$-seminorms will always be Hausdorff.

If $T \colon b \longmapsto (T_k(b))_{k \in \mathbb{N}}$ is a differential seminorm on $\mathcal{B}$, it is easy to see that setting \begin{equation} \label {submdiff} s_m(b) = \sum_{k = 0}^m T_k(b) \qquad b \in \mathcal{B},\end{equation} produces an increasing sequence $(s_m)_{m \in \mathbb{N}}$ of submultiplicative $*$-seminorms on $\mathcal{B}$ generating the same topology as the original sequence $(T_k)_{k \in \mathbb{N}}$.

With all of these preliminaries out of the way, let us now come to concrete realizations of the structures discussed in this section by function algebras equipped with the deformed product.

\section{Rieffel's Function Algebras} \label{rieffel}

Let $\mathcal{C}$ be a C$^*$-algebra. Define $\mathcal{S}^\mathcal{C}(\mathbb{R}^n)$ as the space of $\mathcal{C}$-valued Schwartz functions or, in other words, the $\mathcal{C}$-valued smooth functions on $\mathbb{R}^n$ which, together with all of their partial derivatives, are rapidly decreasing at infinity. Also, define $\mathcal{B}^\mathcal{C}(\mathbb{R}^n)$ as the space of $\mathcal{C}$-valued bounded smooth functions on $\mathbb{R}^n$ whose partial derivatives of all orders are also bounded (when $\mathcal{C} = \mathbb{C}$, we will write simply $\mathcal{S}(\mathbb{R}^n)$ and $\mathcal{B}(\mathbb{R}^n)$, respectively).

We can define two ``$L^2$-type'' norms on $\mathcal{S}^\mathcal{C}(\mathbb{R}^n)$, namely \begin{equation} \|f\|_{L^2} := \left( \int_{\mathbb{R}^n} \|f(x)\|_\mathcal{C}^2 \, dx \right)^{1/2} = \left( \int_{\mathbb{R}^n} \|f(x)^* f(x)\|_\mathcal{C} \, dx \right)^{1/2}, \qquad f \in \mathcal{S}^\mathcal{C}(\mathbb{R}^n), \end{equation} and \begin{equation} \|f\|_2 := \left\| \int_{\mathbb{R}^n} f(x)^*f(x) \, dx \right\|_\mathcal{C}^{1/2}, \qquad f \in \mathcal{S}^\mathcal{C}(\mathbb{R}^n). \end{equation} Clearly, $\|f\|_2 \leqslant \|f\|_{L^2}$, for all $f \in \mathcal{S}^\mathcal{C}(\mathbb{R}^n)$. The Banach space completion $E_n$ of $\mathcal{S}^\mathcal{C}(\mathbb{R}^n)$ with respect to the norm $\|\, \cdot \,\|_2$ possesses the structure of a \textit{Hilbert C$^*$-module} \cite{lance}, with subjacent $\mathcal{C}$-valued inner product \cite[p.~2]{lance} obtained as the continuous extension of the map \begin{equation*} (f, g) \longmapsto \int_{\mathbb{R}^n} f(x)^*g(x) \, dx, \qquad (f, g) \in \mathcal{S}^\mathcal{C}(\mathbb{R}^n) \times \mathcal{S}^\mathcal{C}(\mathbb{R}^n) \end{equation*} to $E_n \times E_n$. This $\mathcal{C}$-valued inner product will be denoted in what follows by $\langle \, \cdot \,, \cdot \, \rangle_{E_n}$ or, when there is no risk of confusion, simply by $\langle \, \cdot \,, \cdot \, \rangle$. Following \cite[p.~9]{lance}, we will denote the C$^*$-algebra of (bounded) \textit{adjointable operators} on the Hilbert $\mathcal{C}$-module $E_n$, equipped with the usual operator C$^*$-norm $\|\, \cdot \,\|$, by $\mathcal{L}_\mathcal{C}(E_n)$.

In order to say a few words about the Banach space completion of $\mathcal{S}^\mathcal{C}(\mathbb{R}^n)$ with respect to the norm $\|\, \cdot \,\|_{L^2}$, we first need to fix some notations: if $\lambda$ denotes the Lebesgue measure on $\mathbb{R}^n$, we will say that a function $f \colon \mathbb{R}^n \longrightarrow \mathcal{C}$ is $\lambda$-simple if $f(x) = \sum_{j = 1}^N 1_{B_j}(x) \, c_j$, for some fixed $N > 0$ and all $x \in \mathbb{R}^n$, where $c_j$ are elements of $\mathcal{C}$ and $1_{B_j}$ are indicator functions of Borel-measurable subsets $B_j$ of $\mathbb{R}^n$ such that $\lambda(B_j) < + \infty$, for all $1 \leqslant j \leqslant N$ \cite[Definition 1.1.13, p.~8]{analysis-bochner}. Moreover, we will say that a function $f \colon \mathbb{R}^n \longrightarrow \mathcal{C}$ is strongly $\lambda$-measurable if it is the $\lambda$-almost everywhere pointwise limit of a sequence of $\lambda$-simple functions \cite[Definition 1.1.14, p.~8]{analysis-bochner}. With these terminologies in mind, we define $L^2(\mathbb{R}^n, \mathcal{C})$ as the space of equivalence classes of strongly $\lambda$-measurable square-integrable $\mathcal{C}$-valued functions on $\mathbb{R}^n$ \cite[Definition 1.2.15, p.~21]{analysis-bochner}, which is the Banach space completion of $\mathcal{S}^\mathcal{C}(\mathbb{R}^n)$ with respect to the norm $\|\, \cdot \,\|_{L^2}$. In fact, as noted in Lemma \ref{embedding} of the Appendix, one may adapt the proof of Lemma \ref{toscano-merklen1} to show that $\mathcal{S}^\mathcal{C}(\mathbb{R}^n)$ is dense in $(L^2(\mathbb{R}^n, \mathcal{C}), \|\, \cdot \,\|_{L^2})$. The space $L^2(\mathbb{R}^n, \mathcal{C})$ is continuously embedded in $E_n$ as a dense subspace, a fact which will play an important role in the proof of Theorem \ref{calderonvaillancourt}; see Appendix \ref{appendixd}.

On the other hand, $\mathcal{S}^\mathcal{C}(\mathbb{R}^n)$ and $\mathcal{B}^\mathcal{C}(\mathbb{R}^n)$ become Fr\'echet spaces when equipped with the sequences of norms defined by \begin{equation} \label{normss} \|f\|_{\mathcal{S}^\mathcal{C}, m} = \max_{|\alpha| \leqslant m} \sup_{x \in \mathbb{R}^n} (1 + |x|^2)^{m/2} \|\partial^\alpha f(x)\|_\mathcal{C}, \qquad f \in \mathcal{S}^\mathcal{C}(\mathbb{R}^n), \, m \in \mathbb{N} \end{equation} and \begin{equation} \label{normsb} \|f\|_{\mathcal{B}^\mathcal{C}, m} = \max_{|\alpha| \leqslant m} \sup_{x \in \mathbb{R}^n} \|\partial^\alpha f(x)\|_\mathcal{C}, \qquad f \in \mathcal{B}^\mathcal{C}(\mathbb{R}^n), \, m \in \mathbb{N}, \end{equation} respectively (we shall use the simplified symbol $|\, \cdot \,|$ for the standard Euclidean norm and, below, a dot for the standard Euclidean scalar product in $\mathbb{R}^n$: $|x| := (\sum_{k = 1}^n |x_k|^2)^{1/2}$, $x \cdot y := \sum_{k = 1}^n x_k \, y_k$).

Now fix a skew-symmetric linear transformation $J$ on $\mathbb{R}^n$ and $f \in \mathcal{B}^\mathcal{C}(\mathbb{R}^n)$. Then it is shown in reference \cite{rieffel} that the linear operator defined by the (iterated) integral \begin{equation} \label{rieffelop} L_f(g)(x) := \int_{\mathbb{R}^n} \left( \int_{\mathbb{R}^n} f(x + Ju) \, g(x + v) \, e^{2 \pi i u \cdot v} \, dv \right) du, \qquad g \in \mathcal{S}^\mathcal{C}(\mathbb{R}^n), \, x \in \mathbb{R}^n, \end{equation} maps $\mathcal{S}^\mathcal{C}(\mathbb{R}^n)$ into $\mathcal{S}^\mathcal{C}(\mathbb{R}^n)$ \cite[Proposition 3.3, p.~25]{rieffel}, satisfies $\langle L_f(g), h \rangle = \langle g, L_{f^*}(h) \rangle$, for all $g, h \in \mathcal{S}^\mathcal{C}(\mathbb{R}^n)$, \cite[Proposition 4.2, p.~30]{rieffel} and extends to a bounded operator on the Hilbert $\mathcal{C}$-module $E_n$ \cite[Theorem 4.6 \& Corollary 4.7, p.~34]{rieffel}. By the continuity of the $\mathcal{C}$-valued inner product we see that this extension, also denoted by $L_f$, is an adjointable operator on $E_n$ satisfying $(L_f)^* = L_{f^*}$. Moreover, for all $f_1, f_2 \in \mathcal{B}^\mathcal{C}(\mathbb{R}^n)$, we have the identity \begin{equation} \label{interplay} L_{f_1} L_{f_2} = L_{f_1 \times_J f_2}, \end{equation} where $\times_J$ is \textit{Rieffel's deformed product} \cite[p.~23]{rieffel}, defined by the (oscillatory) integral \begin{equation} \label{rieffelprod} (f_1 \times_J f_2)(x) := \int_{\mathbb{R}^n} \int_{\mathbb{R}^n} f_1(x + Ju) \, f_2(x + v) \, e^{2 \pi i u \cdot v} \, dv \, du, \qquad x \in \mathbb{R}^n. \end{equation} Actually, as will be discussed in more detail below, the operator $L_f$ is a pseudodifferential operator with symbol $(x, \xi) \longmapsto f(x - J\xi / (2 \pi))$. The interplay given by Equation \eqref{interplay} between the algebra of pseudodifferential operators $L_f$'s and the algebra $\mathcal{B}^\mathcal{C}(\mathbb{R}^n)$, equipped with the product $\times_J$, motivates the following definition:

\begin{definition} \label{rieffelalgebras} The function algebras obtained by equipping the Fr\'echet spaces $\mathcal{S}^\mathcal{C}(\mathbb{R}^n)$ and $\mathcal{B}^\mathcal{C}(\mathbb{R}^n)$ with the deformed product $\times_J$ above, instead of the usual pointwise product, and with the involution operation defined pointwise, via the involution of $\mathcal{C}$, will be denoted by $\mathcal{S}_J^\mathcal{C}(\mathbb{R}^n)$ and $\mathcal{B}_J^\mathcal{C}(\mathbb{R}^n)$, respectively. Also, $\overline{\mathcal{S}_J^\mathcal{C}(\mathbb{R}^n)}$ and $\overline{\mathcal{B}_J^\mathcal{C}(\mathbb{R}^n)}$ will denote their completions with respect to the operator C$^*$-norms $\|\, \cdot \,\|_{\mathcal{S}_J^\mathcal{C}}$ and $\|\, \cdot \,\|_{\mathcal{B}_J^\mathcal{C}}$, respectively, which are defined via the faithful $*$-homomorphism $f \longmapsto L_f$ \cite[Definition 4.8, p.~35]{rieffel} of $\mathcal{B}_J^\mathcal{C}(\mathbb{R}^n)$ into $\mathcal{L}_\mathcal{C}(E_n)$: \begin{equation*} \|f\|_{\mathcal{S}_J^\mathcal{C}} := \|L_f\|, \text{ for } f \in \mathcal{S}_J^\mathcal{C}(\mathbb{R}^n) \quad \text{and} \quad \|f\|_{\mathcal{B}_J^\mathcal{C}} := \|L_f\|, \text{ for } f \in \mathcal{B}_J^\mathcal{C}(\mathbb{R}^n); \end{equation*} note that $\|\, \cdot \,\|_{\mathcal{S}_J^\mathcal{C}}$ is just the restriction of $\|\, \cdot \,\|_{\mathcal{B}_J^\mathcal{C}}$ to $\mathcal{S}_J^\mathcal{C}(\mathbb{R}^n)$. Accordingly, the $*$-algebras of pseudodifferential operators $\mathcal{S}_J^\mathcal{C}$ and $\mathcal{B}_J^\mathcal{C}$, with the usual multiplication given by composition, are defined by $\mathcal{S}_J^\mathcal{C} := \left\{L_f: f \in \mathcal{S}_J^\mathcal{C}(\mathbb{R}^n)\right\}$ and $\mathcal{B}_J^\mathcal{C} := \left\{L_f: f \in \mathcal{B}_J^\mathcal{C}(\mathbb{R}^n)\right\}$. \end{definition}

Observe that all of the Fr\'echet $*$-algebras $\mathcal{B}_J^\mathcal{C}(\mathbb{R}^n)$ are represented on the \textit{same} module $E_n$, \textit{independently} of the skew-symmetric linear transformation $J$ on $\mathbb{R}^n$. We also caution the reader not to confuse the operator C$^*$-norms $\|\, \cdot \,\|_{\mathcal{S}_J^\mathcal{C}}$ and $\|\, \cdot \,\|_{\mathcal{B}_J^\mathcal{C}}$ with the ``sup norms'' $\|\, \cdot \,\|_{\mathcal{S}^\mathcal{C}, 0}$ and $\|\, \cdot \,\|_{\mathcal{B}^\mathcal{C}, 0}$. In fact, nothing guarantees, for a general $J$, that these sup norms are C$^*$-norms with respect to the deformed product $\times_J$. Of course, they are when $J = 0$, because then the deformed product reduces to the usual pointwise product given by $(fg)(x) := f(x) g(x)$, for all $x \in \mathbb{R}^n$ \cite[Corollary 2.8, p.~13]{rieffel} and, if in addition, $\mathcal{C}$ is the field $\mathbb{C}$ of complex numbers, then $\mathcal{S}_J^\mathcal{C}(\mathbb{R}^n)$ and $\mathcal{B}_J^\mathcal{C}(\mathbb{R}^n)$ are just the usual commutative Fr\'echet $*$-algebras of complex-valued functions, with $\|\, \cdot \,\|_{\mathcal{S}_J^\mathcal{C}} = \|\, \cdot \,\|_{\mathcal{S}^\mathcal{C}, 0}$ and $\|\, \cdot \,\|_{\mathcal{B}_J^\mathcal{C}} = \|\, \cdot \,\|_{\mathcal{B}^\mathcal{C}, 0}$; later, in Proposition \ref{sup-op}, we will extend these equalities of norms to the case when $\mathbb{C}$ is replaced by a general C$^*$-algebra $\mathcal{C}$. But for a general $J$, we expect these equalities to break down, and so one of our main concerns in what follows will be to construct a sequence of $*$-norms generating the topology of $\mathcal{B}_J^\mathcal{C}(\mathbb{R}^n)$ which is well-behaved with respect to the deformed product $\times_J$.

In the remainder of this section, we will first construct a differential norm $T \colon f \longmapsto (T_k(f))_{k \in \mathbb{N}}$ on $\mathcal{B}_J^\mathcal{C}(\mathbb{R}^n)$ generating its natural Fr\'echet topology and satisfying $T_0 = \|\, \cdot \,\|_{\mathcal{B}_J^\mathcal{C}}$. As corollaries, we will show existence of a unique C$^*$-norm on $\mathcal{B}_J^\mathcal{C}(\mathbb{R}^n)$ and the property of spectral invariance of $\mathcal{B}_J^\mathcal{C}(\mathbb{R}^n)$ in its C$^*$-completion. These results will be derived under the assumption that $\mathcal{C}$ is unital, but uniqueness of the C$^*$-norm will then be shown to hold even when $\mathcal{C}$ is not unital. Once we are done with the algebra $\mathcal{B}_J^\mathcal{C}(\mathbb{R}^n)$, we will adapt some of our results to obtain similar corollaries for the algebra $\mathcal{S}_J^\mathcal{C}(\mathbb{R}^n)$.

\subsection*{Pseudodifferential operators with \texorpdfstring{$\mathcal{C}$}{C}-valued symbols}$\ $

$\ $

Let $\mathcal{C}$ be a C$^*$-algebra. In order to attain some of our goals, we will use features of Lie group representation theory for the Heisenberg group of dimension $2n + 1$, defined as \begin{equation*} H_{2n + 1}(\mathbb{R}) = \left\{\begin{bmatrix}
1 & \texttt{a}^T & c\\
0 & I_n & -\texttt{b}\\
0 & 0 & 1
\end{bmatrix}: \texttt{a}, \texttt{b} \in \mathbb{R}^n, c \in \mathbb{R}\right\}, \end{equation*} where the product is just standard matrix multiplication and $I_n$ denotes the identity matrix of $M_n(\mathbb{R})$. It admits a strongly continuous unitary representation $U$ on the Hilbert $\mathcal{C}$-module $E_n$ given by \begin{equation*} U_{\texttt{a}, \texttt{b}, c}(f)(x) := U\begin{bmatrix}
1 & \texttt{a}^T & c\\
0 & I_n & -\texttt{b}\\
0 & 0 & 1
\end{bmatrix}(f)(x) := e^{ic} e^{i \texttt{b} \cdot x} f(x - \texttt{a}), \qquad f \in \mathcal{S}^\mathcal{C}(\mathbb{R}^n), \, x \in \mathbb{R}^n, \end{equation*} where the term ``unitary'' is in the sense of Hilbert C$^*$-modules \cite[p.~24]{lance} from which we can construct a corresponding ``adjoint'' representation of the Heisenberg group $H_{2n + 1}(\mathbb{R})$ on the C$^*$-algebra of adjointable operators $\mathcal{L}_\mathcal{C}(E_n)$ by \begin{equation*} \text{Ad}\,U \colon \begin{bmatrix}
1 & \texttt{a}^T & c\\
0 & I_n & -\texttt{b}\\
0 & 0 & 1
\end{bmatrix} \longmapsto (\text{Ad}\,U)(\texttt{a}, \texttt{b}, c)(\, \cdot \,) := U_{\texttt{a}, \texttt{b}, c} \, (\, \cdot \,) \, (U_{\texttt{a}, \texttt{b}, c})^{-1}. \end{equation*} Note that $(\text{Ad}\,U)(\texttt{a}, \texttt{b}, c)$ does not depend on the real variable $c$ -- so we will simply write $(\text{Ad}\,U)(\texttt{a}, \texttt{b})$ -- and that $\text{Ad}\,U$, in contrast to $U$, is not strongly continuous; this means that the C$^*$-subalgebra $C(\text{Ad}\,U)$ of continuous elements for $\text{Ad}\,U$ is, in general, properly contained in $\mathcal{L}_\mathcal{C}(E_n)$. Next, let $C^\infty(\text{Ad}\,U)$ be the Fr\'echet $*$-algebra of smooth elements for the representation $\text{Ad}\,U$. Denoting by $\delta_j$ the $j^\text{th}$ infinitesimal generator of the representation $\text{Ad}\,U$, $1 \leqslant j \leqslant 2n$, we have that $\delta_j(A) = \partial_j[(\text{Ad}\,U)(\texttt{a}, \texttt{b})(A)]|_{\texttt{a} = \texttt{b} = 0}$, for all $A$ belonging to $C^\infty(\text{Ad}\,U)$. The Fr\'echet topology on $C^\infty(\text{Ad}\,U)$ is defined by the family \begin{equation} \label{smoothfrechet} \left\{\rho_m: m \in \mathbb{N}\right\} \end{equation} of norms, where \begin{equation*} \rho_0(A) := \|A\|, \quad \delta_0 := I \end{equation*} and \begin{equation*} \rho_m(A) := \max \left\{\|(\delta_{i_1} \ldots \delta_{i_m})A\|: 0 \leqslant i_j \leqslant 2n\right\}, \qquad A \in C^\infty(\text{Ad}\,U), \, m \geqslant 1. \end{equation*}

Working with pseudodifferential operators involves the Fourier transform $\mathcal{F}$, sometimes also denoted by $\hat{\, \cdot \,}$ and defined by \begin{equation*} \mathcal{F}(g)(\xi) := \frac{1}{(2 \pi)^{n / 2}} \int_{\mathbb{R}^n} e^{-i s \cdot \xi} \, g(s) \, ds, \qquad g \in \mathcal{S}^\mathcal{C}(\mathbb{R}^n), \, \xi \in \mathbb{R}^n. \end{equation*} It is a continuous linear operator on the Fr\'echet space $\mathcal{S}^\mathcal{C}(\mathbb{R}^n)$. The same is true for the inverse Fourier transform $\mathcal{F}^{-1}$ on $\mathcal{S}^\mathcal{C}(\mathbb{R}^n)$, which is defined by $\mathcal{F}^{-1}(g)(x) := \mathcal{F}(g)(-x)$. For more details about the Fourier transform, see \cite[Proposition 2.4.22, p.~117]{analysis-bochner}; for general facts about Bochner integrals, see \cite[Chapter 1]{analysis-bochner}. We shall also use the following generalized version of Plancherel's Theorem for $E_n$, which follows from Fubini's Theorem. For any $u, v \in \mathcal{S}^\mathcal{C}(\mathbb{R}^n)$, we have \begin{align} \label{plancherel} (2 \pi)^{n/2} \langle \mathcal{F}(u), v \rangle &= \int_{\mathbb{R}^n} \left(\int_{\mathbb{R}^n} e^{-i x \cdot y} u(y) \, dy\right)^* v(x) \, dx \\ &= \int_{\mathbb{R}^n} u(y)^* \left(\int_{\mathbb{R}^n} e^{i x \cdot y} v(x) \, dx\right) dy = (2 \pi)^{n/2} \langle u, \mathcal{F}^{-1}(v) \rangle, \nonumber \end{align} just as in \cite[Proposi\c c\~ao B.3]{merklentese}; substituting $v = \mathcal{F}(u)$ in the above equality shows that $\mathcal{F}$ uniquely extends by continuity to an isometry on $E_n$. By the continuity of the $\mathcal{C}$-valued inner product, we see that $\langle \mathcal{F}(u), v \rangle = \langle u, \mathcal{F}^{-1}(v) \rangle$ also holds for $u, v \in E_n$. In particular, we get that $\mathcal{F}$ is an adjointable operator on $E_n$ with $\mathcal{F}^* = \mathcal{F}^{-1}$.

Given $a \in \mathcal{B}^\mathcal{C}(\mathbb{R}^{2n})$ one may define a pseudodifferential operator $\text{Op}(a) \colon \mathcal{S}^\mathcal{C}(\mathbb{R}^n) \longrightarrow \mathcal{S}^\mathcal{C}(\mathbb{R}^n)$ by \begin{equation} \label{op_0} \text{Op}(a)(g)(x) := \frac{1}{(2 \pi)^{n / 2}} \int_{\mathbb{R}^n} e^{i x \cdot \xi} \, a(x, \xi) \, \widehat{g}(\xi) \, d\xi \end{equation} or, more explicitly, by the (iterated) integral \begin{equation} \label{op} \text{Op}(a)(g)(x) := \frac{1}{(2 \pi)^n} \int_{\mathbb{R}^n} \left( \int_{\mathbb{R}^n} e^{i (x - y) \cdot \xi} \, a(x, \xi) \, g(y) \, dy \right) d\xi, \end{equation} for all $x \in \mathbb{R}^n$. A simple calculation shows that every $L_f \in \mathcal{B}_J^\mathcal{C}$ may be written this way, with $a(x, \xi) = f(x - J\xi / (2 \pi))$.

To obtain some information about how the above pseudodifferential operators with $\mathcal{C}$-valued symbols are related to the adjoint action of the Heisenberg group, we need a version of the Calder\'on-Vaillancourt Theorem for the Hilbert C$^*$-module $E_n$. The proof of such a version is the content of \cite[Theorem 2.1]{merklen}, but there seems to be a mistake in the proof, more precisely in the integration by parts at the bottom of page 1281. For this reason, we will give a new proof of that result (which is Theorem \ref{calderonvaillancourt}, below), and with the additional benefit that we do not need to restrict ourselves to \textit{separable} C$^*$-algebras $\mathcal{C}$. The proof of Theorem \ref{calderonvaillancourt} below is based on \cite[Theorem 3.14]{seiler} and on \cite[Cap\'itulo 3]{merklentese}. Just as in \cite[p.~169]{seiler}, for any given $\beta \in \mathbb{N}^n$ and $x \in \mathbb{R}^n$, we adopt the notations \begin{equation*} (i + x)^\beta := \prod_{j = 1}^n (i + x_j)^{\beta_j}, \quad (i + x)^{- \beta} := [(i + x)^\beta]^{-1} \quad \text{and}\end{equation*} \begin{equation*} D_{x_j} := -i \frac{\partial}{\partial x_j}, \quad (i + D_x)^\beta := \prod_{j = 1}^n (i + D_{x_j})^{\beta_j}.\end{equation*}

\begin{theorem} \label{calderonvaillancourt} Let $\mathcal{C}$ be a C$^*$-algebra (unital, or not). Then $\text{Op}(a)$ extends to a bounded operator on $E_n$, for every $a \in \mathcal{B}^\mathcal{C}(\mathbb{R}^{2n})$. More precisely, denoting by $\mathring{\alpha} \in \mathbb{N}^n$ the multiindex $(1, \ldots, 1)$, there exists a constant $C > 0$ such that for every $a \in \mathcal{B}^\mathcal{C}(\mathbb{R}^{2n})$ we have the estimate $\|\text{Op}(a)\| \leqslant C \pi(a)$, where $\pi(a)$ is defined by \begin{equation} \pi(a) := \max_{\beta, \gamma \leqslant \mathring{\alpha}} \sup \left\{\|\partial_x^\beta \partial_\xi^\gamma a(x, \xi)\|_\mathcal{C}: x, \xi \in \mathbb{R}^n\right\}. \end{equation} \end{theorem}

\begin{proof} Write $A := \text{Op}(a)$, in order to simplify the notation. Suppose first that $a$ is a compactly supported smooth function on $\mathbb{R}^n$, and fix $u, v \in \mathcal{S}^\mathcal{C}(\mathbb{R}^n)$. Let us calculate $\langle \hat{v}, \widehat{Au} \rangle$ noting that, under these hypotheses, we can make free use of Fubini's Theorem and integrate by parts to get \begin{align*} \langle \hat{v}, \widehat{Au} \rangle &= \frac{1}{(2 \pi)^{3n/2}} \int \! \int \! \int \! \int e^{-i x \cdot \eta} \hat{v}(\eta)^* [(i + x - y)^{- \mathring{\alpha}} (i + D_\xi)^{\mathring{\alpha}} e^{i \xi \cdot (x - y)}] \, a(x, \xi) \, u(y) \, d\xi \, dx \, dy \, d\eta \\ &= \frac{1}{(2 \pi)^{3n/2}} \int \! \int \! \int \! \int [(i + \xi - \eta)^{-\mathring{\alpha}}(i + D_x)^{\mathring{\alpha}} e^{i x \cdot (\xi - \eta)}] \, e^{-i \xi \cdot y} (i + x - y)^{- \mathring{\alpha}} \hat{v}(\eta)^* \\ & \hspace{7.5cm} [(i - D_\xi)^{\mathring{\alpha}} a(x, \xi)] u(y) \, d\xi \, dx \, dy \, d\eta \\ &= \frac{1}{(2 \pi)^{3n/2}} \int \! \int \! \int \! \int (i + \xi - \eta)^{-\mathring{\alpha}} e^{i x \cdot (\xi - \eta)} e^{-i \xi \cdot y} \hat{v}(\eta)^* F(y, x, \xi) \, u(y) \, d\xi \, dx \, dy \, d\eta, \end{align*} where $F(y, x, \xi) := (i - D_x)^{\mathring{\alpha}} \left\{(i + x - y)^{- \mathring{\alpha}} [(i - D_\xi)^{\mathring{\alpha}} a(x, \xi)]\right\}$. From the formula \begin{equation*} (i - D_x)^{\mathring{\alpha}}(wz) = \sum_{\gamma \leqslant \mathring{\alpha}} (-1)^{|\gamma|} [(i - D_x)^{\mathring{\alpha} - \gamma} w] D_x^\gamma z, \qquad w \in C^\infty(\mathbb{R}^n), \, z \in \mathcal{B}^\mathcal{C}(\mathbb{R}^n), \end{equation*} we see that \begin{equation} \label{c-s} \langle \hat{v}, \widehat{Au} \rangle = \frac{1}{(2 \pi)^{3n/2}} \sum_{\gamma \leqslant \mathring{\alpha}} \int \! \int e^{i x \cdot \xi} g(x, \xi) [D_x^\gamma (i - D_\xi)^{\mathring{\alpha}} a(x, \xi)] f_\gamma(x, \xi) \, dx \, d\xi, \end{equation} with \begin{align*} f_\gamma(x, \xi) &:= (-1)^{|\gamma|} \int e^{-i \xi \cdot y} (i - D_x)^{\mathring{\alpha} - \gamma} (i + x - y)^{- \mathring{\alpha}} u(y) \, dy, \\ g(x, \xi) &:= \int e^{-i x \cdot \eta} (i + \xi - \eta)^{-\mathring{\alpha}} \hat{v}(\eta)^* \, d\eta.\end{align*} Estimating the expression in Equation \eqref{c-s} will be based on the Cauchy-Schwarz inequality for Hilbert C$^*$-modules \cite[Proposition 1.1, p.~3]{lance}. But first, we want to prove that the functions $g$ and $f_\gamma$ so defined all belong to $L^2(\mathbb{R}^{2n}, \mathcal{C})$ (and hence to $E_{2n}$), so we proceed just as in \cite[Lema 3.17]{merklentese}. Fix $(x, \xi) \in \mathbb{R}^{2n}$. Using the equality \begin{equation*} e^{-i x \cdot \eta} = \frac{(1 - \Delta_\eta)^N}{(1 + |x|^2)^N} \, e^{-i x \cdot \eta}, \qquad \Delta_\eta := \sum_{k = 1}^n \frac{\partial^2}{\partial \eta_k^2}, \qquad N \in \mathbb{N}, \end{equation*} for every $\eta \in \mathbb{R}^n$, and integrating by parts the expression which defines $g$ gives, for every $N \in \mathbb{N}$, the formula \begin{equation*} g(x, \xi) = \frac{1}{(1 + |x|^2)^N} \int e^{-i x \cdot \eta} (1 - \Delta_\eta)^N [(i + \xi - \eta)^{-\mathring{\alpha}} \hat{v}(\eta)^*] \, d\eta. \end{equation*} Therefore, after successive applications of the Leibniz product rule we may write $g(x, \xi)$ as a linear combination of terms of the form \begin{equation*} \frac{1}{(1 + |x|^2)^N} \int e^{-i x \cdot \eta} (i + \xi - \eta)^{- \beta} \partial_\eta^{\beta'} \hat{v}(\eta)^* \, d\eta, \qquad \beta, \beta' \in \mathbb{N}^n, \, \beta \geqslant \mathring{\alpha}, \, \beta' \geqslant 0. \end{equation*} Using Peetre's inequality \cite[(3.6)]{kohn-nirenberg} \begin{equation*} |(i + \xi - \eta)^{- \beta}| = \prod_{k = 1}^n \frac{1}{(1 + |\xi_k - \eta_k|^2)^{\beta_k / 2}} \leqslant 2^{|\beta| / 2} \prod_{k = 1}^n \frac{(1 + |\eta_k|^2)^{\beta_k / 2}}{(1 + |\xi_k|^2)^{\beta_k / 2}} \end{equation*} we obtain, for each $\beta, \beta' \in \mathbb{N}^n$, that the (C$^*$-)norm $\|\, \cdot \,\|_\mathcal{C}$ evaluated on the corresponding term is bounded from above by \begin{equation*} 2^{|\beta| / 2} \frac{1}{(1 + |x|^2)^N} \prod_{k = 1}^n \frac{1}{(1 + |\xi_k|^2)^{\beta_k / 2}} \int (1 + |\eta_k|^2)^{\beta_k / 2} \|\partial_\eta^{\beta'} \hat{v}(\eta)^*\|_\mathcal{C} \, d\eta. \end{equation*} Since $\beta_k \geqslant 1$, for all $1 \leqslant k \leqslant n$, we may choose $N > n/4$ to finally conclude that $g$ belongs to $L^2(\mathbb{R}^{2n}, \mathcal{C})$. By an analogous reasoning, one sees that the same conclusion holds for the functions $f_\gamma$. This implies that, for each $\gamma \leqslant \mathring{\alpha}$, the function $(x, \xi) \longmapsto [D_x^\gamma (i - D_\xi)^{\mathring{\alpha}} a(x, \xi)] f_\gamma(x, \xi)$ also belongs to $L^2(\mathbb{R}^{2n}, \mathcal{C})$, so applying the Cauchy-Schwarz inequality for Hilbert C$^*$-modules yields \begin{equation} \label{cauchyschwartzcalderon} \|\langle \hat{v}, \widehat{Au} \rangle\|_\mathcal{C} \leqslant \frac{1}{(2 \pi)^{3n/2}} \sum_{\gamma \leqslant \mathring{\alpha}} \underbrace{\|g^*\|_2}_{\text{(I)}} \, \underbrace{\|(x, \xi) \longmapsto [D_x^\gamma (i - D_\xi)^{\mathring{\alpha}} a(x, \xi)] f_\gamma(x, \xi)\|_2}_{\text{(II)}}. \end{equation} Let us first estimate (I). Define the function $h(\xi) := (i + \xi)^{- \mathring{\alpha}}$ and, for each fixed $\xi \in \mathbb{R}^n$, define $h_\xi(\eta) := (i + \xi - \eta)^{- \mathring{\alpha}}$, for all $\eta \in \mathbb{R}^n$, so that \begin{equation*} g^*(x, \xi) = \int e^{i x \cdot \eta} \overline{h_\xi(\eta)} \hat{v}(\eta) \, d\eta = (2 \pi)^{n/2} \mathcal{F}(\overline{h_\xi} \cdot \hat{v})(-x) =: G_\xi(x), \end{equation*} for every $x, \xi \in \mathbb{R}^n$. Then \begin{align} \label{estimate_h} & \int_{\mathbb{R}^n} \int_{\mathbb{R}^n} g(x, \xi) g(x, \xi)^* \, dx \, d\xi = \int_{\mathbb{R}^n} \langle G_\xi, G_\xi \rangle_{E_n} \, d\xi = (2 \pi)^n \int_{\mathbb{R}^n} \langle \overline{h_\xi} \cdot \hat{v}, \overline{h_\xi} \cdot \hat{v} \rangle_{E_n} \, d\xi \\ &= \, (2 \pi)^n \int_{\mathbb{R}^n} \left[\int_{\mathbb{R}^n} |h_\xi|^2(\eta) \, \hat{v}(\eta)^* \hat{v}(\eta) \, d\eta\right] d\xi = (2 \pi)^n \int_{\mathbb{R}^n} |h(\xi)|^2 \, d\xi \int_{\mathbb{R}^n} \hat{v}(\eta)^* \hat{v}(\eta) \, d\eta, \nonumber \end{align} so $\|g^*\|_2 = C_1 \|v\|_2$, where $C_1 := (2 \pi)^{n/2} (\int_{\mathbb{R}^n} |h(\xi)|^2 \, d\xi)^{1/2}$. To estimate (II), note that for every positive linear functional $\rho$ on $\mathcal{C}$ and $c, d \in \mathcal{C}$ we have $\rho(d^*c^*cd) \leqslant \|c^*c\|_\mathcal{C} \, \rho(d^*d)$ \cite[Theorem 3.3.7, p.~90]{murphy}, which implies \begin{align*} & \int_{\mathbb{R}^n} \int_{\mathbb{R}^n} \rho \left(f_\gamma(x, \xi)^* [D_x^\gamma (i - D_\xi)^{\mathring{\alpha}} a(x, \xi)]^* [D_x^\gamma (i - D_\xi)^{\mathring{\alpha}} a(x, \xi)] f_\gamma(x, \xi)\right) dx \, d\xi \\ &\leqslant \int_{\mathbb{R}^n} \int_{\mathbb{R}^n} \|D_x^\gamma (i - D_\xi)^{\mathring{\alpha}} a(x, \xi)\|_\mathcal{C}^2 \, \rho \left(f_\gamma(x, \xi)^* f_\gamma(x, \xi)\right) dx \, d\xi \\ &\leqslant \left[\sup_{x, \xi} \left\{\|D_x^\gamma (i - D_\xi)^{\mathring{\alpha}} a(x, \xi)\|_\mathcal{C}\right\}\right]^2 \int_{\mathbb{R}^n} \int_{\mathbb{R}^n} \rho \left(f_\gamma(x, \xi)^* f_\gamma(x, \xi)\right) dx \, d\xi. \end{align*} Therefore \cite[Theorem 3.4.3, p.~95]{murphy}, \begin{align*} & \int_{\mathbb{R}^n} \int_{\mathbb{R}^n} \left\{[D_x^\gamma (i - D_\xi)^{\mathring{\alpha}} a(x, \xi)] f_\gamma(x, \xi)\right\}^* [D_x^\gamma (i - D_\xi)^{\mathring{\alpha}} a(x, \xi)] f_\gamma(x, \xi) dx \, d\xi \\ &\leqslant \left[\sup_{x, \xi} \left\{\|D_x^\gamma (i - D_\xi)^{\mathring{\alpha}} a(x, \xi)\|_\mathcal{C}\right\}\right]^2 \int_{\mathbb{R}^n} \int_{\mathbb{R}^n} f_\gamma(x, \xi)^* f_\gamma(x, \xi) dx \, d\xi, \end{align*} which finally gives us the estimate \begin{equation} \label{calderon_h} \|[D_x^\gamma (i - D_\xi)^{\mathring{\alpha}} a(x, \xi)] f_\gamma(x, \xi)\|_2 \leqslant C_2 \pi(a) \|f_\gamma\|_2, \end{equation} where $C_2 > 0$ is independent of $a$. But \begin{equation*} f_\gamma(x, \xi) = (-1)^{|\gamma|} \int e^{-i \xi \cdot y} (i - D_x)^{\mathring{\alpha} - \gamma} (i + x - y)^{- \mathring{\alpha}} u(y) \, dy = (-1)^{|\gamma|} (2 \pi)^{n/2} \mathcal{F}((i - D_x)^{\mathring{\alpha} - \gamma} h_x \cdot u)(\xi), \end{equation*} for each fixed $x, \xi \in \mathbb{R}^n$. Analogously as in Equation \eqref{estimate_h} we get \begin{equation} \label{estimate_h_2} \|f_\gamma\|_2 \leqslant C_2' \|h\|_2 \|u\|_2, \qquad C_2' > 0. \end{equation} Finally, combining Equations \eqref{cauchyschwartzcalderon}, \eqref{estimate_h}, \eqref{calderon_h} and \eqref{estimate_h_2} gives \begin{equation} \label{compactsupp} \|\langle v, Au \rangle\|_\mathcal{C} = \|\langle \hat{v}, \widehat{Au} \rangle\|_\mathcal{C} \leqslant K \pi(a) \|v\|_2 \, \|u\|_2, \end{equation} for some constant $K > 0$ which is independent of $a$, $u$ and $v$.

Now we turn to the general case where $a \in \mathcal{B}^\mathcal{C}(\mathbb{R}^{2n})$. Let $u, v \in \mathcal{S}^\mathcal{C}(\mathbb{R}^n)$, $0 \leqslant \phi \leqslant 1$ be a compactly supported smooth function on $\mathbb{R}^{2n}$ which equals 1 on a neighborhood of 0 and define, for each $m \in \mathbb{N} \backslash \left\{0\right\}$, the function \begin{equation*} a_m(x, \xi) := \phi \left( \frac{x}{m}, \frac{\xi}{m} \right) \, a(x, \xi), \qquad (x, \xi) \in \mathbb{R}^{2n}. \end{equation*} We are going to show that $\langle v, \text{Op}(b_m)u \rangle$ goes to 0 as $m \rightarrow + \infty$, where $b_m := a - a_m$. Since an application of the Cauchy-Schwarz inequality implies $\|\langle v, \text{Op}(b_m)u \rangle\|_\mathcal{C} \leqslant \|v\|_2 \|\text{Op}(b_m)u\|_2$, it suffices to show that $\|\text{Op}(b_m)u\|_2$ converges to 0, as $m \rightarrow + \infty$. First, note that since $\|e^{i x \cdot \xi} b_m(x, \xi) \hat{u}(\xi)\|_\mathcal{C} \leqslant \sup \left\{\|a(x, \xi)\|_\mathcal{C}: (x, \xi) \in \mathbb{R}^{2n}\right\} \|\hat{u}(\xi)\|_\mathcal{C}$, for every fixed $(x, \xi) \in \mathbb{R}^{2n}$, it follows from the Dominated Convergence Theorem that $[\text{Op}(b_m)u](x)$ converges to 0, as $m \rightarrow + \infty$, for every fixed $x \in \mathbb{R}^n$. Also, there exists a constant $C' > 0$ which is independent of $m$ and of $a$ such that $\pi(a_m) \leqslant C' \pi(a)$, so we have the estimates \begin{align*} \|[\text{Op}(b_m)u](x)\|_\mathcal{C} &\leqslant (2 \pi)^{-n/2} |(i + x)^{- \mathring{\alpha}}| \int_{\mathbb{R}^n} \|(i - D_\xi)^{\mathring{\alpha}}[b_m(x, \xi) \hat{u}(\xi)]\|_\mathcal{C} \, d\xi \\ &\leqslant M \pi(b_m) |(i + x)^{- \mathring{\alpha}}| \leqslant (C' + 1) M \pi(a) |(i + x)^{- \mathring{\alpha}}|,\end{align*} where $M > 0$ depends on the numbers $\int_{\mathbb{R}^n} \|D_\xi^\beta \hat{u}(\xi)\|_\mathcal{C} \, d\xi$, with $\beta \leqslant \mathring{\alpha}$, but does not depend on $a$, $m$ or $x$. Therefore, another application of the Dominated Convergence Theorem finally establishes that $\langle v, \text{Op}(b_m)u \rangle$ goes to 0, as $m \rightarrow + \infty$.

Substituting $A = \text{Op}(a_m)$ on \eqref{compactsupp} we get \begin{equation*} \|\langle v, \text{Op}(a_m)u \rangle\|_\mathcal{C} \leqslant K \pi(a_m) \|v\|_2 \, \|u\|_2 \leqslant K C' \pi(a) \|v\|_2 \, \|u\|_2. \end{equation*} Taking the limit $m \rightarrow + \infty$ on both sides of this inequality gives $\|\langle v, \text{Op}(a)u \rangle\|_\mathcal{C} \leqslant K C' \pi(a) \|v\|_2 \, \|u\|_2$. Since $u, v \in \mathcal{S}^\mathcal{C}(\mathbb{R}^n)$ are arbitrary, this actually shows that there exists $C > 0$ such that $\|\text{Op}(a)u\|_2 \leqslant C \pi(a) \|u\|_2$, for all $a \in \mathcal{B}^\mathcal{C}(\mathbb{R}^{2n})$ and $u \in \mathcal{S}^\mathcal{C}(\mathbb{R}^n)$. \end{proof}

We note that not only does $\text{Op}(a)$ extends to a bounded operator on $E_n$, but this extension is also an adjointable operator on $E_n$. For the convenience of the reader, we will now give a quick proof of this fact, which is inspired by the exposition in \cite[Cap\'itulo 4]{merklentese}. This observation is important, because the representation $\text{Ad}\,U$ of the Heisenberg group $H_{2n + 1}(\mathbb{R})$ is implemented by $*$-automorphisms on the C$^*$-algebra of adjointable operators $\mathcal{L}_\mathcal{C}(E_n)$, and later it will be convenient to treat $\mathcal{B}_J^\mathcal{C}$ (see Definition \ref{rieffelalgebras}) as a $*$-subalgebra of $\mathcal{L}_\mathcal{C}(E_n)$.

\begin{proposition} Let $\mathcal{C}$ be a C$^*$-algebra (unital, or not) and let $a \in \mathcal{B}^\mathcal{C}(\mathbb{R}^{2n})$. Then $\text{Op}(a)$ is an adjointable operator on $E_n$ with $[\text{Op}(a)]^* = \text{Op}(p)$, for a certain $p \in \mathcal{B}^\mathcal{C}(\mathbb{R}^{2n})$. \end{proposition}

\begin{proof} First, assume that $a$ belongs to the space $C_c^\infty(\mathbb{R}^{2n}, \mathcal{C})$ of $\mathcal{C}$-valued compactly supported smooth functions on $\mathbb{R}^{2n}$. We are going to prove that there exists $p \in \mathcal{S}^\mathcal{C}(\mathbb{R}^{2n})$ satisfying $\langle \text{Op}(a)u, v \rangle = \langle u, \text{Op}(p)v \rangle$, for every $u, v \in \mathcal{S}^\mathcal{C}(\mathbb{R}^n)$. An application of Fubini's Theorem shows that \begin{align*} \langle \text{Op}(a) u, v \rangle &= \int_{\mathbb{R}^n} \left[ \frac{1}{(2 \pi)^n} \int_{\mathbb{R}^n} \int_{\mathbb{R}^n} e^{i (x - y) \cdot \xi} \, a(x, \xi) \, u(y) \, dy \, d\xi \right]^* v(x) \, dx \\ &= \int_{\mathbb{R}^n} u(y)^* \left[ \frac{1}{(2 \pi)^n} \int_{\mathbb{R}^n} \int_{\mathbb{R}^n} e^{i (y - x) \cdot \xi} \, a(x, \xi)^* \, v(x) \, d\xi \, dx \right] dy, \end{align*} for all $u, v \in \mathcal{S}^\mathcal{C}(\mathbb{R}^n)$. Define $c \in \mathcal{S}^\mathcal{C}(\mathbb{R}^{2n})$ by \begin{equation*} c(y, z) := \frac{1}{(2 \pi)^{n/2}} \int_{\mathbb{R}^n} e^{i z \cdot \xi} \, a(y - z, \xi)^* \, d\xi, \qquad y, z \in \mathbb{R}^n, \end{equation*} and define the function $p \in \mathcal{S}^\mathcal{C}(\mathbb{R}^{2n})$ by $p(y, z) := \mathcal{F} (\xi \longmapsto c(y, \xi))(z)$, so that $c(y, z) = \mathcal{F}^{-1} (\xi \longmapsto p(y, \xi))(z)$, for all $y, z \in \mathbb{R}^n$. Then \begin{equation*} \frac{1}{(2 \pi)^{n/2}} \int_{\mathbb{R}^n} e^{i (y - x) \cdot \xi} a(x, \xi)^* \, d\xi = c(y, y - x) = \frac{1}{(2 \pi)^{n/2}} \int_{\mathbb{R}^n} e^{i (y - x) \cdot \xi} p(y, \xi) \, d\xi,\end{equation*} so \begin{align*} \langle \text{Op}(a) u, v \rangle &= \int_{\mathbb{R}^n} u(y)^* \left[\frac{1}{(2 \pi)^n} \int_{\mathbb{R}^n} \int_{\mathbb{R}^n} e^{i (y - x) \cdot \xi} a(x, \xi)^* \, v(x) \, d\xi \, dx \right] dy \\ &= \int_{\mathbb{R}^n} u(y)^* \left[\frac{1}{(2 \pi)^n} \int_{\mathbb{R}^n} \int_{\mathbb{R}^n} e^{i (y - x) \cdot \xi} p(y, \xi) \, v(x) \, dx \, d\xi \right] dy = \langle u, \text{Op}(p) v \rangle, \end{align*} for all $u, v \in \mathcal{S}^\mathcal{C}(\mathbb{R}^n)$. Therefore, the equality $\langle \text{Op}(a)u, v \rangle = \langle u, \text{Op}(p) v \rangle$, for every $u, v \in E_n$, follows from a continuity argument. An easy calculation gives the following identity \begin{align*} p(y, \xi) &= \frac{1}{(2 \pi)^n} \int_{\mathbb{R}^n} e^{-i z \cdot \xi} \left[\int_{\mathbb{R}^n} e^{i z \cdot \eta} \, a(y - z, \eta)^* \, d\eta \right] dz \\ &= \frac{1}{(2 \pi)^n} \int_{\mathbb{R}^n} \int_{\mathbb{R}^n} e^{-i z \cdot \eta} \, a(y - z, \xi - \eta)^* \, dz \, d\eta, \qquad (y, \xi) \in \mathbb{R}^{2n}, \end{align*} which will be useful in the next step of the proof.

Suppose, now, that $a \in \mathcal{B}^\mathcal{C}(\mathbb{R}^{2n})$. Then employing the definition of oscillatory integrals in \cite[pp.~66-69]{cordes} (where they are called \textit{finite part integrals}), we define \begin{equation} \label{involution} p(y, \xi) := \frac{1}{(2 \pi)^n} \int_{\mathbb{R}^n} \int_{\mathbb{R}^n} e^{-i z \cdot \eta} \, a(y - z, \xi - \eta)^* \, dz \, d\eta \end{equation} \begin{equation*} := \frac{1}{(2 \pi)^n} \int_{\mathbb{R}^n} \int_{\mathbb{R}^n} e^{-i z \cdot \eta} \, (1 + |z|^2)^{-N} (1 - \Delta_\eta)^N \Bigl\{ (1 + |\eta|^2)^{-M} (1 - \Delta_z)^M \Bigl[ a(y - z, \xi - \eta)^* \Bigr] \Bigr\} dz \, d\eta, \end{equation*} where $M, N$ are fixed positive integers which are chosen in order to turn the right-hand side integral into an absolutely convergent one (it suffices to take $M, N > n/2$; also, the above definition is independent of $M$ and $N$). Then differentiating under the integral sign shows that $p$ belongs to $\mathcal{B}^\mathcal{C}(\mathbb{R}^{2n})$. The above definition of oscillatory integral is essentially the same as the one employed by Rieffel in his monograph \cite[Proposition 1.6, p.~6]{rieffel}.

Let $0 \leqslant \phi \leqslant 1$ be a compactly supported smooth function on $\mathbb{R}^n$ which equals 1 on a neighborhood of 0 and define, for each $a \in \mathcal{B}^\mathcal{C}(\mathbb{R}^{2n})$ and $j \in \mathbb{N} \backslash \left\{0\right\}$, the $\mathcal{C}$-valued compactly supported smooth functions on $\mathbb{R}^{2n}$ \begin{equation*} \phi_j(x) := \phi \left(\frac{x}{j}\right), \quad a_j(x, \xi) := a(x, \xi) \, \phi_j(x) \, \phi_j(\xi) \quad \text{and} \end{equation*} $p_j(y, \xi) := \frac{1}{(2 \pi)^n} \int_{\mathbb{R}^n} \int_{\mathbb{R}^n} e^{-i z \cdot \eta} \, a_j(y - z, \xi - \eta)^* \, dz \, d\eta$, for all $(y, \xi) \in \mathbb{R}^{2n}$. Then using the definition of oscillatory integrals in Equation \eqref{involution} we get, after an application of Fubini's Theorem and of the Dominated Convergence Theorem, the equality $\lim_{j \rightarrow + \infty} p_j(y, \xi) = p(y, \xi)$, for every $(y, \xi) \in \mathbb{R}^{2n}$. Hence, \begin{align*} \langle u, \text{Op}(p) v \rangle &= \int_{\mathbb{R}^n} u(y)^* \left[ \frac{1}{(2 \pi)^{n/2}} \int_{\mathbb{R}^n} e^{i y \cdot \xi} \, \lim_{j \rightarrow + \infty} p_j(y, \xi) \, \hat{v}(\xi) \, d\xi \right] dy \\ &= \lim_{j \rightarrow + \infty} \frac{1}{(2 \pi)^{n/2}} \int_{\mathbb{R}^{2n}} u(y)^* e^{i y \cdot \xi} \, p_j(y, \xi) \, \hat{v}(\xi) \, d\xi \, dy = \lim_{j \rightarrow + \infty} \langle u, \text{Op}(p_j) v \rangle,\end{align*} for all $u, v \in \mathcal{S}^\mathcal{C}(\mathbb{R}^n)$. By an analogous reasoning, \begin{equation*} \lim_{j \rightarrow + \infty} \langle \text{Op}(a_j)u, v \rangle = \lim_{j \rightarrow + \infty} \int_{\mathbb{R}^n} \left[ \frac{1}{(2 \pi)^{n/2}} \int_{\mathbb{R}^n} e^{- i x \cdot \xi} \, \hat{u}(\xi)^* a_j(x, \xi)^* \, d\xi \right] v(x) \, dx = \langle \text{Op}(a)u, v \rangle, \end{equation*} so $\langle \text{Op}(a)u, v \rangle = \lim_{j \rightarrow + \infty} \langle \text{Op}(a_j)u, v \rangle = \lim_{j \rightarrow + \infty} \langle u, \text{Op}(p_j) v \rangle = \langle u, \text{Op}(p) v \rangle$, for all $u, v \in \mathcal{S}^\mathcal{C}(\mathbb{R}^n)$. By a continuity argument, $\langle \text{Op}(a)u, v \rangle = \langle u, \text{Op}(p) v \rangle$, for all $u, v \in E_n$, so $\text{Op}(a)$ is an adjointable operator on $E_n$, for every $a \in \mathcal{B}^\mathcal{C}(\mathbb{R}^{2n})$, as claimed. \end{proof}

We finish this subsection by noting that $\left\{ \text{Op}(a): a \in \mathcal{B}^\mathcal{C}(\mathbb{R}^{2n}) \right\}$ is actually a $*$-subalgebra of $\mathcal{L}_\mathcal{C}(E_n)$. In fact, using the more suggestive notation $a^\dagger$ to denote the function $p$ defined in Equation \eqref{involution}, we see that the restriction of the involution and composition maps to $\left\{ \text{Op}(a): a \in \mathcal{B}^\mathcal{C}(\mathbb{R}^{2n}) \right\}$ are given, respectively, by $\text{Op}(a) \longmapsto \text{Op}(a^\dagger)$ and $\text{Op}(a) \circ \text{Op}(b) \longmapsto \text{Op}(a \times b)$, where \begin{equation} a^\dagger(x, \xi) := \frac{1}{(2 \pi)^n} \int_{\mathbb{R}^n} \int_{\mathbb{R}^n} e^{-i z \cdot \eta} \, a(x - z, \xi - \eta)^* \, dz \, d\eta \end{equation} \begin{equation*} := \frac{1}{(2 \pi)^n} \int_{\mathbb{R}^n} \int_{\mathbb{R}^n} e^{-i z \cdot \eta} \, (1 + |z|^2)^{-N} (1 - \Delta_\eta)^N \Bigl\{ (1 + |\eta|^2)^{-M} (1 - \Delta_z)^M \Bigl[ a(x - z, \xi - \eta)^* \Bigr] \Bigr\} dz \, d\eta, \end{equation*} and \begin{equation} (a \times b)(x, \xi) := \frac{1}{(2 \pi)^n} \int_{\mathbb{R}^n} \int_{\mathbb{R}^n} e^{-i z \cdot \eta} \, a(x, \xi - \eta) \, b(x - z, \xi) \, dz \, d\eta \end{equation} \begin{equation*} := \frac{1}{(2 \pi)^n} \int_{\mathbb{R}^n} \int_{\mathbb{R}^n} e^{-i z \cdot \eta} \, (1 + |z|^2)^{-N} (1 - \Delta_\eta)^N \Bigl\{ (1 + |\eta|^2)^{-M} (1 - \Delta_z)^M \Bigl[ a(x, \xi - \eta) \, b(x - z, \xi) \Bigr] \Bigr\} dz \, d\eta, \end{equation*} for all $x, \xi \in \mathbb{R}^n$ (for the scalar case $\mathcal{C} = \mathbb{C}$, see \cite[Proposition 4.2, p.~64]{cordes} and \cite[Theorem 4.7, p.~68]{cordes}). Just as in \eqref{involution}, it suffices to take integers $M, N > n/2$, with the above definitions also being independent of $M$ and $N$; differentiating under the integral sign shows at once that $a \times b$ belongs to $\mathcal{B}^\mathcal{C}(\mathbb{R}^{2n})$. Moreover, we see by the associativity of the composition operation that $\text{Op}((a \times b) \times c) = \text{Op}(a \times (b \times c))$, for all $a, b, c \in \mathcal{B}^\mathcal{C}(\mathbb{R}^{2n})$. We will now show that the linear map $\text{Op} \colon a \longmapsto \text{Op}(a)$ on $\mathcal{B}^\mathcal{C}(\mathbb{R}^{2n})$ is injective, from which it will follow that $\times$ is also an associative operation. Let $(e_\alpha)_{\alpha \in \Gamma}$ be an approximate identity for $\mathcal{C}$. Then for every scalar-valued function $g \in \mathcal{S}(\mathbb{R}^n)$, we have that the function $g_\alpha := g \cdot e_\alpha \colon x \longmapsto g(x) \cdot e_\alpha$ belongs to $\mathcal{S}^\mathcal{C}(\mathbb{R}^n)$, for every $\alpha \in \Gamma$. Hence, the hypothesis $\text{Op}(a) = 0$ implies \begin{equation*} 0 = \rho(\text{Op}(a)(g_\alpha)(x)) = \frac{1}{(2 \pi)^{n / 2}} \int_{\mathbb{R}^n} e^{i x \cdot \xi} \, \rho(a(x, \xi) \, e_\alpha) \, \widehat{g}(\xi) \, d\xi, \end{equation*} for all $g \in \mathcal{S}(\mathbb{R}^n)$, $x \in \mathbb{R}^n$ and every continuous linear functional $\rho$ on $\mathcal{C}$. But then injectivity of the map $b \longmapsto \text{Op}(b)$ for scalar-valued symbols $b \in \mathcal{B}(\mathbb{R}^n)$ \cite[p.~220]{melomerklen2} proves that the map $(x, \xi) \longmapsto \rho(a(x, \xi) \, e_\alpha)$ must be identically zero, for every continuous linear functional $\rho$ on $\mathcal{C}$. Therefore, by Hahn-Banach's Theorem we obtain $a(x, \xi) \, e_\alpha = 0$, for every fixed $(x, \xi) \in \mathbb{R}^{2n}$ and $\alpha \in \Gamma$, so taking the limit in $\alpha$ shows that $a$ must be identically zero, as wanted. As a corollary, since $\text{Op}((a \times b) \times c - a \times (b \times c)) = 0$, for all $a, b, c \in \mathcal{B}^\mathcal{C}(\mathbb{R}^{2n})$, it is also true that $(a \times b) \times c = a \times (b \times c)$, for all $a, b, c \in \mathcal{B}^\mathcal{C}(\mathbb{R}^{2n})$.

We can also use a similar argument to obtain the associativity of Rieffel's deformed product, which we show next. In fact, as noted in the remark following Equation \eqref{op}, if $J$ is a fixed skew-symmetric linear transformation on $\mathbb{R}^n$, then every operator $L_f$, for $f \in \mathcal{B}^\mathcal{C}(\mathbb{R}^n)$, may be written as $\text{Op}(\tilde{f})$, where $\tilde{f}(x, \xi) := f(x - J\xi / (2 \pi))$, for all $x, \xi \in \mathbb{R}^n$ (note that the map $\tilde{\cdot} \colon f \longmapsto \tilde{f}$ \textit{depends} on the fixed $J$). Therefore, \begin{align*} (\tilde{f} \times \tilde{g})(x, \xi) & := \frac{1}{(2 \pi)^n} \int_{\mathbb{R}^n} \int_{\mathbb{R}^n} e^{-i z \cdot \eta} \, (1 + |z|^2)^{-N} (1 - \Delta_\eta)^N \\ & \hspace{1.7cm} \Bigl\{ (1 + |\eta|^2)^{-M} (1 - \Delta_z)^M \Bigl[ f(x - J(\xi - \eta) / (2 \pi)) \, g(x - z - J\xi / (2 \pi)) \Bigr] \Bigr\} dz \, d\eta, \end{align*} which shows that (i) $(\tilde{f} \times \tilde{g})(x, \xi) = (f \times_J g)(x - J\xi / (2 \pi)) = \widetilde{(f \times_J g)}(x, \xi)$ and, in particular, (ii) $(\tilde{f} \times \tilde{g})(x, 0) = (f \times_J g)(x) = \widetilde{(f \times_J g)}(x, 0)$, for all $f, g \in \mathcal{B}^\mathcal{C}(\mathbb{R}^n)$, $x, \xi \in \mathbb{R}^n$. Note that for all $f, g \in \mathcal{B}^\mathcal{C}(\mathbb{R}^n)$ and $\phi \in \mathcal{S}^\mathcal{C}(\mathbb{R}^n)$ we have the equality \begin{align*} (f \times_J g) \times_J \phi & = L_{f \times_J g}(\phi) = \text{Op}(\widetilde{f \times_J g})(\phi) = \text{Op}(\tilde{f} \times \tilde{g})(\phi) \\ & = [\text{Op}(\tilde{f}) \circ \text{Op}(\tilde{g})](\phi) = [L_f \circ L_g](\phi) = f \times_J (g \times_J \phi) \end{align*} (since $J$ is arbitrary, this incidentally gives an alternative proof for the relation \eqref{interplay}). Hence, since for all $f, g, h \in \mathcal{B}^\mathcal{C}(\mathbb{R}^n)$ and $\phi \in \mathcal{S}^\mathcal{C}(\mathbb{R}^n)$ we have $f \times_J g \in \mathcal{B}^\mathcal{C}(\mathbb{R}^n)$ and $h \times_J \phi \in \mathcal{S}^\mathcal{C}(\mathbb{R}^n)$ (the former relation follows from (ii), while the latter is a consequence of the equality $\text{Op}(\tilde{h})(\phi) = h \times_J \phi$), we obtain \begin{align*} L_{(f \times_J g) \times_J h}(\phi) & = [(f \times_J g) \times_J h] \times_J \phi = (f \times_J g) \times_J (h \times_J \phi) = f \times_J [g \times_J (h \times_J \phi)] \\ & = f \times_J [(g \times_J h) \times_J \phi] = [f \times_J (g \times_J h)] \times_J \phi = L_{f \times_J (g \times_J h)}(\phi). \end{align*} Since $J$ and $\phi$ are arbitrary and the map $\text{Op}$ is injective, we have obtained the desired result. We note that, although the formulas in these final two paragraphs will not be used as tools to derive any of our main results, they have been included here to provide a more transparent link between Rieffel's deformed product and the composition of pseudodifferential operators.

\subsection*{The algebra \texorpdfstring{$\mathcal{B}_J^\mathcal{C}(\mathbb{R}^n)$}{BJC(Rn)}}$\ $

$\ $

For every pseudodifferential operator $\text{Op}(a)$, $a \in \mathcal{B}^\mathcal{C}(\mathbb{R}^{2n})$, and every $\alpha, \beta \in \mathbb{N}^n$ one has \begin{equation} \label{derivsymb} \partial_\texttt{a}^\alpha \partial_\texttt{b}^\beta[(\text{Ad}\,U)(\texttt{a}, \texttt{b})(\text{Op}(a))] = (-1)^{|\alpha| + |\beta|}(\text{Ad}\,U)(\texttt{a}, \texttt{b})(\text{Op}(\partial_x^\alpha \partial_\xi^\beta a)), \qquad \texttt{a}, \texttt{b} \in \mathbb{R}^n.\end{equation} Indeed, by Theorem \ref{calderonvaillancourt} we have the estimate \begin{equation} \label{calderon} \|\text{Op}(a)\| \leqslant C \max_{|\gamma|, |\delta| \leqslant n} \sup \left\{\|\partial_x^\gamma \partial_\xi^\delta a(x, \xi)\|_\mathcal{C}: x, \xi \in \mathbb{R}^n\right\},\end{equation} where $C > 0$ is independent of $a$ (see also \cite[Corollary 4.7, p.~34]{rieffel} for a particular version of these inequalities adapted for the operators $L_f$, $f \in \mathcal{B}_J^\mathcal{C}(\mathbb{R}^n)$). Therefore, denoting by $e_j$ the $j^\text{th}$ element of the canonical basis of $\mathbb{R}^n$, the equalities \begin{equation*} a((x, \xi) + h (e_k, 0)) - a(x, \xi) - h \, \frac{\partial a}{\partial x_k} (x, \xi) = h^2 \int_0^1 t \int_0^1 \frac{\partial^2 a}{\partial x_k^2} ((x, \xi) + t s h (e_k, 0)) ds \, dt,\end{equation*} and \begin{equation*} a((x, \xi) + h (0, e_k)) - a(x, \xi) - h \, \frac{\partial a}{\partial \xi_k} (x, \xi) = h^2 \int_0^1 t \int_0^1 \frac{\partial^2 a}{\partial \xi_k^2} ((x, \xi) + t s h (0, e_k)) ds \, dt,\end{equation*} $h \in \mathbb{R}$, $1 \leqslant k \leqslant n$, combined with the estimate \eqref{calderon}, give \eqref{derivsymb} in the case $|\alpha| + |\beta| = 1$. The equality for general $\alpha, \beta \in \mathbb{N}^n$ follows from an iteration of this procedure. It shows that the operator $\text{Op}(a)$ belongs to the $*$-algebra $C^\infty(\text{Ad}\,U)$ of smooth elements for the representation $\text{Ad}\,U$, for every $a \in \mathcal{B}^\mathcal{C}(\mathbb{R}^{2n})$. In particular, every element of $\mathcal{B}_J^\mathcal{C}$ is contained in $C^\infty(\text{Ad}\,U)$, so we may equip $\mathcal{B}_J^\mathcal{C}$ with the subspace topology induced by the usual Fr\'echet topology of $C^\infty(\text{Ad}\,U)$, which will be denoted by $\tau_{\mathcal{B}_J^\mathcal{C}, \text{C}^\infty}$. Also, injectivity of the map $L \colon f \longrightarrow L_f$ allows us to equip $\mathcal{B}_J^\mathcal{C}$ with a Fr\'echet space topology $\tau_\mathcal{B}$ induced by the natural topology of the function algebra $\mathcal{B}_J^\mathcal{C}(\mathbb{R}^n)$ defined by the family \eqref{normsb} of $*$-norms. Then \eqref{calderon} combined with \eqref{derivsymb} shows that \begin{equation} \rho_m(\text{Op}(a)) \leqslant C \max_{|\gamma|, |\delta| \leqslant n + m} \sup \left\{\|\partial_x^\gamma \partial_\xi^\delta a(x, \xi)\|_\mathcal{C}: x, \xi \in \mathbb{R}^n\right\}, \qquad a \in \mathcal{B}^\mathcal{C}(\mathbb{R}^{2n}), \, m \in \mathbb{N},\end{equation} for the same constant $C$ above (for the definition of $\rho_m$, see \eqref{smoothfrechet}). When specialized to the operators $L_f$ (and to functions $f$ defined on $\mathbb{R}^n$, instead of $\mathbb{R}^{2n}$) this gives \begin{equation} \label{finer} \rho_m(L_f) \leqslant \tilde{C}_m \max_{|\gamma| \leqslant n + m} \sup \left\{\|\partial_x^\gamma f(x)\|_\mathcal{C}: x \in \mathbb{R}^n\right\}, \qquad f \in \mathcal{B}_J^\mathcal{C}(\mathbb{R}^n), \, m \in \mathbb{N},\end{equation} which implies that $\tau_\mathcal{B}$ is finer than $\tau_{\mathcal{B}_J^\mathcal{C}, \text{C}^\infty}$ (note that $\tilde{C}_m$ depends also on the linear transformation $J$).

It is not clear that $\mathcal{B}_J^\mathcal{C}$, when equipped with $\tau_{\mathcal{B}_J^\mathcal{C}, \text{C}^\infty}$, is a closed subspace of $C^\infty(\text{Ad}\,U)$. To see that this is indeed the case, we are going to resort to the ``symbol map'' $S$ constructed in reference \cite{melomerklen2}. Consequently, we will need to temporarily assume that $\mathcal{C}$ is a \textit{unital} C$^*$-algebra. We make the important observation that the results of \cite[Section 2]{melomerklen2} which will be invoked, in what follows, are valid for \textit{any} unital C$^*$-algebra, and do not require the separability assumption made in that reference. For a more explicit discussion on this issue, we refer the reader to Appendix \ref{appendixc}, where in particular we show that $E_{2n}$ can be identified with an \textit{interior tensor product} $E_n \otimes E_n$ (see Lemma \ref{toscano-merklen1}); this is used in the definition of the map $S$ described next.

Consider the surjective map \cite[Theorem 1]{melomerklen2} \begin{equation*} S \colon C^\infty(\text{Ad}\,U) \longrightarrow \mathcal{B}^\mathcal{C}(\mathbb{R}^{2n}) \end{equation*} given by \begin{equation} \label{S} S(A)(x, \xi) := (2 \pi)^{n/2} \langle u \cdot 1_\mathcal{C}, \left\{ (D \, [(\text{Ad}\,U)(-x, -\xi)(A)] \, \mathcal{F}^{-1}) \otimes I_{E_n} \right\} \, v \cdot 1_\mathcal{C} \rangle_{E_{2n}}, \end{equation} for all $A \in C^\infty(\text{Ad}\,U)$ and $(x, \xi) \in \mathbb{R}^{2n}$, where $D := \prod_{j = 1}^n (1 + \partial_{x_j})^2 (1 + \partial_{\xi_j})^2$ and $u$ and $v$ are (fixed) suitable scalar-valued functions belonging to $L^2(\mathbb{R}^{2n}) \cap L^1(\mathbb{R}^{2n})$, which are independent of $A$ (for the definitions of $u$ and $v$, see the statement of Lemma \ref{toscano-merklen2}; for a description of the embedding $L^2(\mathbb{R}^n) \hookrightarrow E_n$, see Appendix \ref{appendixd}). Then the composition $S \circ Op$ is the identity operator on $\mathcal{B}^\mathcal{C}(\mathbb{R}^{2n})$ so that, in particular, $S(L_f) = \tilde{f}$, where $\tilde{f}(x, \xi) := f(x - J \xi/(2 \pi))$. Applying the Cauchy-Schwarz inequality for Hilbert C$^*$-modules to \eqref{S} yields an estimate in the opposite direction of the one given by the Calder\'on-Vaillancourt-type inequality \eqref{calderon}, namely, for all $A \in C^\infty(\text{Ad}\,U)$ and $a = S(A)$, \begin{equation} \label{invcalderon} \sup \left\{\|a(x, \xi)\|_\mathcal{C}: x, \xi \in \mathbb{R}^n\right\} \leqslant (2 \pi)^{n/2} \|u\|_2 \|v\|_2 \|D \, [(\text{Ad}\,U)(-x, -\xi)(A)]|_{x = \xi = 0}\|, \end{equation} where we have used that $\|\mathcal{F}^{-1}\| = 1$ and that $\text{Ad}\,U$ is a representation by $*$-automorphisms on $\mathcal{L}_\mathcal{C}(E_n)$. So just as the estimate \eqref{calderon} gives a bound for the operator norm of $\text{Op}(a)$ in terms of sup-norms of derivatives of $a$, the estimate \eqref{invcalderon} provides a bound for the sup-norm of $a = S(A)$ in terms of operator norms of derivatives of $A$. Using \eqref{derivsymb} after substituting $A = \text{Op}(b)$ on \eqref{invcalderon}, where $b := \partial_x^\gamma \partial_\xi^\delta a$, $\gamma, \delta \in \mathbb{N}^n$, $a \in \mathcal{B}^\mathcal{C}(\mathbb{R}^{2n})$, gives \begin{equation} \label{invcalderon-2} \sup \left\{\|\partial_x^\gamma \partial_\xi^\delta a(x, \xi)\|_\mathcal{C}: x, \xi \in \mathbb{R}^n\right\} \leqslant (2 \pi)^{n/2} \|u\|_2 \|v\|_2 \|(\partial_x^\gamma \partial_\xi^\delta D)\, [(\text{Ad}\,U)(-x, -\xi)(\text{Op}(a))]|_{x = \xi = 0}\|. \end{equation} Noting that $L_f = \text{Op}(\tilde{f})$, \eqref{invcalderon-2} with $\delta = 0$ immediately implies \begin{align} \label{coarser} \max_{|\gamma| \leqslant m} \sup \left\{\|\partial_x^\gamma f(x)\|_\mathcal{C}: x \in \mathbb{R}^n\right\} & = \max_{|\gamma| \leqslant m} \sup \left\{\|\partial_x^\gamma \tilde{f}(x, 0)\|_\mathcal{C}: x \in \mathbb{R}^n\right\} \\ & \leqslant \max_{|\gamma| \leqslant m} \sup \left\{\|\partial_x^\gamma \tilde{f}(x, \xi)\|_\mathcal{C}: x, \xi \in \mathbb{R}^n\right\} \stackrel{\eqref{invcalderon-2}}{\leqslant} \tilde{E} \, \rho_{2n + m}(L_f), \nonumber \end{align} for all $f \in \mathcal{B}_J^\mathcal{C}(\mathbb{R}^n)$, $m \in \mathbb{N}$ and some constant $\tilde{E} > 0$ which does not depend on $f$ or $J$ (indeed, we may choose $\tilde{E} = 16^n (2 \pi)^{n/2} \|u\|_2 \|v\|_2$), showing that $\tau_\mathcal{B}$ is coarser than $\tau_{\mathcal{B}_J^\mathcal{C}, \text{C}^\infty}$. Therefore, we conclude that $\tau_\mathcal{B} = \tau_{\mathcal{B}_J^\mathcal{C}, \text{C}^\infty}$.

\begin{remark} Note that there is a certain ``uniformity'' aspect in the estimate \eqref{coarser}: the constant $\tilde{E}$ that shows up does not depend on the seminorms under consideration. \end{remark}

Having proved the equality of the topologies $\tau_\mathcal{B}$ and $\tau_{\mathcal{B}_J^\mathcal{C}, \text{C}^\infty}$, we are in a good position to define an appropriate differential norm on $\mathcal{B}_J^\mathcal{C}(\mathbb{R}^n)$ (see Definition \ref{diffnorm}):

\begin{theorem} \label{b-diffnorm} Let $\mathcal{C}$ be a unital C$^*$-algebra and $J$ be a skew-symmetric linear transformation on $\mathbb{R}^n$. The topologies of the pseudodifferential operator algebra $\mathcal{B}_J^\mathcal{C}$ and the function algebra $\mathcal{B}_J^\mathcal{C}(\mathbb{R}^n)$ are generated by differential norms. In particular, they are Arens-Michael $*$-algebras. \end{theorem}

\begin{proof} It is clear that defining \begin{equation} \label{diffnorms-b} T_0(L_f) := \|L_f\|, \quad T_k(L_f) := \frac{1}{k!} \, \sum_{|\alpha| = k} \|\delta^\alpha(L_f)\|, \quad L_f \in \mathcal{B}_J^\mathcal{C}, \, k \geqslant 1, \, \alpha \in \mathbb{N}^{2n}, \end{equation} where the $\delta^\alpha$'s are the monomials in the generators of the representation $\text{Ad}\,U$, yields a differential norm $T \colon L_f \longmapsto (T_k(L_f))_{k \in \mathbb{N}}$ on $\mathcal{B}_J^\mathcal{C}$ which generates the Fr\'echet topology $\tau_{\mathcal{B}_J^\mathcal{C}, \text{C}^\infty} = \tau_\mathcal{B}$. Moreover, the family $(s_m)_{m \in \mathbb{N}}$ of submultiplicative $*$-norms defined in Equation \eqref{submdiff} generates this same topology. Since the above differential norm on $\mathcal{B}_J^\mathcal{C}$ may be pulled back to a differential norm on $\mathcal{B}_J^\mathcal{C}(\mathbb{R}^n)$ by the $*$-isomorphism $L \colon \mathcal{B}_J^\mathcal{C}(\mathbb{R}^n) \longrightarrow \mathcal{B}_J^\mathcal{C}$, all of the conclusions just stated for $\mathcal{B}_J^\mathcal{C}$ are also true for $\mathcal{B}_J^\mathcal{C}(\mathbb{R}^n)$. \end{proof}

\begin{remark} Assume $\mathcal{C}$ is a non-unital C$^*$-algebra. Since there exists a canonical inclusion $\mathcal{B}_J^\mathcal{C}(\mathbb{R}^n) \hookrightarrow \mathcal{B}_J^{\tilde{\mathcal{C}}}(\mathbb{R}^n)$, we can equip $\mathcal{B}_J^\mathcal{C}(\mathbb{R}^n)$ with the subspace topology induced by $\tau_\mathcal{B} = \tau_{\mathcal{B}_J^{\tilde{\mathcal{C}}}, \text{C}^\infty}$. But the subspace topology defined by the norms in \eqref{normsb} is complete, so $\mathcal{B}_J^\mathcal{C}(\mathbb{R}^n)$ is also an Arens-Michael $*$-algebra in this case. \end{remark}

Before proving Theorem \ref{b-uniq}, we recall the concept of closure under the holomorphic functional calculus:

\begin{definition} \label{holo} \cite[p.~582]{schweitzer} Let $\mathcal{B}$ be a $*$-subalgebra of a C$^*$-algebra $\mathcal{A}$. $\mathcal{B}$ is said to be \textit{closed under the holomorphic functional calculus of $\mathcal{A}$} if, for every element $b$ of $\dot{\mathcal{B}}$ and every holomorphic function $f$ on an open neighborhood $V \subseteq \mathbb{C}$ of $\sigma_{\dot{\mathcal{A}}}(b)$, one has $f(b) \in \dot{\mathcal{B}}$. \end{definition}

The next theorem shows some advantages of dealing with a topology which is generated by a differential (semi)norm:

\begin{theorem} \label{b-uniq} Let $\mathcal{C}$ be a unital C$^*$-algebra and $J$ be a skew-symmetric linear transformation on $\mathbb{R}^n$. Then the algebras $\mathcal{B}_J^\mathcal{C}$ and $\mathcal{B}_J^\mathcal{C}(\mathbb{R}^n)$ are spectrally invariant and closed under the C$^\infty$ and holomorphic functional calculi of their respective C$^*$-completions (which are clearly $*$-isomorphic). They also share the same K-theory of their C$^*$-completions -- more specifically, the inclusion maps induce K-theory isomorphisms. Finally, there exists only one C$^*$-norm on each one of the algebras $\mathcal{B}_J^\mathcal{C}$ and $\mathcal{B}_J^\mathcal{C}(\mathbb{R}^n)$. \end{theorem}

\begin{proof} Spectral invariance follows from \cite[Theorem 3.3 (iii)]{bhattdiff} with $\mathfrak{A}_\tau = \mathcal{B}_J^\mathcal{C}(\mathbb{R}^n)$ and $\phi$ being the identity map, noting that the locally convex topology of $\mathcal{B}_J^\mathcal{C}(\mathbb{R}^n)$ is already complete, while closure under the C$^\infty$-functional calculus follows from \cite[Theorem 3.4]{bhattdiff}. Therefore, as a consequence of Theorem \ref{uniq}, the Fr\'echet $*$-algebra $\mathcal{B}_J^\mathcal{C}(\mathbb{R}^n)$ can be equipped with only one C$^*$-norm, namely, $\|\, \cdot \,\|_{\mathcal{B}_J^\mathcal{C}}$, while the only C$^*$-norm on $\mathcal{B}_J^\mathcal{C}$ is the operator norm $\|\, \cdot \,\|$ of $\mathcal{L}_\mathcal{C}(E_n)$ (see Definition \ref{rieffelalgebras}). By \cite[Lemma 1.2]{schweitzer}, spectral invariance of these algebras in their completions is equivalent to being closed under the holomorphic functional calculus. To prove the statement about the isomorphism of K-theories of $\mathcal{B} := \mathcal{B}_J^\mathcal{C}(\mathbb{R}^n)$ and $\mathcal{A} := \overline{\mathcal{B}_J^\mathcal{C}(\mathbb{R}^n)}$, first note that the group $\text{Inv}(\mathcal{B})$ of invertible elements of $\mathcal{B}$ coincides with $\text{Inv}(\mathcal{A}) \cap \mathcal{B}$, as a result of the spectral invariance claim. Moreover, since the Fr\'echet topology $\tau_{\mathcal{B}_J^{\tilde{\mathcal{C}}}, \text{C}^\infty}$ on $\mathcal{B}$ is finer than the one induced by the C$^*$-topology of $\mathcal{A}$, the inclusion map $i_\mathcal{B} \colon \mathcal{B} \hookrightarrow \mathcal{A}$ is continuous. Therefore, since $\text{Inv}(\mathcal{B})$ is the inverse image of the open set $\text{Inv}(\mathcal{A})$ of $\mathcal{A}$ under $i_\mathcal{B}$, it is open in $\mathcal{B}$. Consequently, it follows from \cite[Proposition 2, p.~113]{waelbroeck} that the inversion map on $\text{Inv}(\mathcal{B})$ is continuous with respect to the (induced) Fr\'echet topology of $\mathcal{B}$. These arguments show that $\mathcal{B}$ is a Fr\'echet algebra with a continuous inversion map on the open set $\text{Inv}(\mathcal{B})$ (thus, a ``\textit{bonne alg\`ebre de Fr\'echet}'', according to \cite[A.1.2, p.~324]{bost}), so the existence of the K-theory isomorphism follows from \cite[Th\'eor\`eme A.2.1, p.~328]{bost}. The conclusion for $\mathcal{B}_J^\mathcal{C}$ is obtained in the same way.\end{proof}

We now extend the uniqueness result regarding C$^*$-norms on $\mathcal{B}_J^\mathcal{C}(\mathbb{R}^n)$ to any C$^*$-algebra $\mathcal{C}$ (unital, or not).

\begin{theorem} \label{b-uniq2} Let $\mathcal{C}$ be a C$^*$-algebra (unital, or not) and $J$ be a skew-symmetric linear transformation on $\mathbb{R}^n$. Then the algebras $\mathcal{B}_J^\mathcal{C}$ and $\mathcal{B}_J^\mathcal{C}(\mathbb{R}^n)$ admit only one C$^*$-norm. \end{theorem}

\begin{proof} It suffices to prove the result for $\mathcal{B}_J^\mathcal{C}$. We begin by noting that $\mathcal{B}_J^\mathcal{C}(\mathbb{R}^n)$ is an ideal in $\mathcal{B}_J^{\tilde{\mathcal{C}}}(\mathbb{R}^n)$: write the function $f \in \mathcal{B}_J^{\tilde{\mathcal{C}}}(\mathbb{R}^n)$ as $x \longmapsto (f_0(x), \lambda(x))$, where $f_0(\, \cdot \,)$ and $\lambda(\, \cdot \,)$ have ranges in $\mathcal{C}$ and $\mathbb{C}$, respectively. Then by the definition of the C$^*$-norm of $\tilde{\mathcal{C}}$, the inequality \begin{equation*} 2 \|\partial^\alpha f(x)\|_{\tilde{\mathcal{C}}} \geqslant \max \left\{\|\partial^\alpha f_0(x)\|_\mathcal{C}, |\partial^\alpha \lambda(x)|\right\} \end{equation*} holds, for all $x \in \mathbb{R}^n$ and $\alpha \in \mathbb{N}^n$, so $f_0$ and $\lambda$ belong to $\mathcal{B}_J^\mathcal{C}(\mathbb{R}^n)$ and $\mathcal{B}_J^\mathbb{C}(\mathbb{R}^n)$, respectively. Thus for every $g \in \mathcal{B}_J^\mathcal{C}(\mathbb{R}^n)$ and $x \in \mathbb{R}^n$ we have \begin{align*} (f \times_J g)(x) &= \int_{\mathbb{R}^n} \int_{\mathbb{R}^n} f(x + Ju) \, g(x + v) \, e^{2 \pi i u \cdot v} \, dv \, du \\ &= \int_{\mathbb{R}^n} \int_{\mathbb{R}^n} (f_0(x + Ju), \lambda(x + Ju)) \, (g(x + v), 0) \, e^{2 \pi i u \cdot v} \, dv \, du \\ &= \int_{\mathbb{R}^n} \int_{\mathbb{R}^n} (f_0(x + Ju) \, g(x + v) + \lambda(x + Ju) \, g(x + v), 0) \, e^{2 \pi i u \cdot v} \, dv \, du, \end{align*} so $f \times_J g$ indeed belongs to $\mathcal{B}_J^\mathcal{C}(\mathbb{R}^n)$ (analogously for $g \times_J f$). Therefore, $\mathcal{B}_J^\mathcal{C}$ is indeed an ideal in $\mathcal{B}_J^{\tilde{\mathcal{C}}}$.

Now let $\|\, \cdot \,\|_0$ be any C$^*$-norm on $\mathcal{B}_J^\mathcal{C}$. Since we know that $\mathcal{B}_J^\mathcal{C}$ is an ideal in $\mathcal{B}_J^{\tilde{\mathcal{C}}}$, the maps \begin{equation} \label{cstarnorm1} \|\, \cdot \,\|_L \colon L_f \longmapsto \sup \left\{\|L_f \circ L_g\|_0: L_g \in \mathcal{B}_J^\mathcal{C}, \|L_g\|_0 \leqslant 1\right\} \end{equation} and \begin{equation} \label{cstarnorm2} \|\, \cdot \,\|_R \colon L_f \longmapsto \sup \left\{\|L_g \circ L_f\|_0: L_g \in \mathcal{B}_J^\mathcal{C}, \|L_g\|_0 \leqslant 1\right\}, \end{equation} are well-defined on $\mathcal{B}_J^{\tilde{\mathcal{C}}}$. Let us show that $\|L_f\|_L = \|L_f\|_R$, for all $L_f \in \mathcal{B}_J^{\tilde{\mathcal{C}}}$. We will give a proof that adapts the strategy of \cite[Lemma 2.1.4, p.~38]{murphy}, which concerns basic facts about the norm of a \textit{double centralizer} on a C$^*$-algebra. For every $L_f \in \mathcal{B}_J^{\tilde{\mathcal{C}}}$ and $L_g, L_h \in \mathcal{B}_J^\mathcal{C}$, we have that $\|(L_g \circ L_f) \circ L_h\|_0 \leqslant \|L_g\|_0 \|L_f\|_R \|L_h\|_0$, so \begin{equation*} \|L_f \circ L_h\|_0 = \sup \left\{\|L_{g'} \circ (L_f \circ L_h)\|_0: L_{g'} \in \mathcal{B}_J^\mathcal{C}, \|L_{g'}\|_0 \leqslant 1\right\} \leqslant \|L_f\|_R \|L_h\|_0, \end{equation*} which implies the inequality $\|L_f\|_L \leqslant \|L_f\|_R$. Similarly, $\|L_g \circ (L_f \circ L_h)\|_0 \leqslant \|L_g\|_0 \|L_f\|_L \|L_h\|_0$, so we obtain \begin{equation*} \|L_g \circ L_f\|_0 = \sup \left\{\|(L_g \circ L_f) \circ L_{h'}\|_0: L_{h'} \in \mathcal{B}_J^\mathcal{C}, \|L_{h'}\|_0 \leqslant 1\right\} \leqslant \|L_g\|_0 \|L_f\|_L \end{equation*} and, consequently, $\|L_f\|_R \leqslant \|L_f\|_L$. Therefore, $\|L_f\|_L = \|L_f\|_R$. Now, we will show that the map $L_f \longmapsto \|L_f\|_L = \|L_f\|_R$ is a C$^*$-norm on $\mathcal{B}_J^{\tilde{\mathcal{C}}}$. To see that the involution is isometric with respect to $\|\, \cdot \,\|_L$, note that $\|L_f^*\|_L \leqslant \|L_f\|_R = \|L_f\|_L$ and $\|L_f\|_L = \|(L_f^*)^*\|_L \leqslant \|L_f^*\|_R = \|L_f^*\|_L$, for every $L_f \in \mathcal{B}_J^{\tilde{\mathcal{C}}}$. On the other hand, to obtain the C$^*$-property for $\|\, \cdot \,\|_L$, first note that taking the supremum on \begin{equation*} \|L_f \circ L_g\|_0^2 = \|(L_g^* \circ L_f^*) \circ (L_f \circ L_g)\|_0 \leqslant \|L_g\|_0 \|(L_f^* \circ L_f) \circ L_g\|_0 \end{equation*} over all $L_g$ satisfying $L_g \in \mathcal{B}_J^\mathcal{C}$ and $\|L_g\|_0 \leqslant 1$, gives $\|L_f\|_L^2 \leqslant \|L_f^* \circ L_f\|_L$, for all $L_f \in \mathcal{B}_J^{\tilde{\mathcal{C}}}$; for the reverse inequality, note that submultiplicativity of $\|\, \cdot \,\|_L$ implies $\|L_f^* \circ L_f\|_L \leqslant \|L_f^*\|_L \|L_f\|_L = \|L_f\|_L^2$. This proves that $\|\, \cdot \,\|_L$ is indeed a C$^*$-norm on $\mathcal{B}_J^{\tilde{\mathcal{C}}}$.

By the uniqueness result proved in Theorem \ref{b-uniq} (applied to $\mathcal{B}_J^{\tilde{\mathcal{C}}}$) we have that \begin{equation*} \|L_f\| = \|L_f\|_L \leqslant \|L_f\|_0, \qquad L_f \in \mathcal{B}_J^\mathcal{C}. \end{equation*} Moreover, if $0 \neq L_f \in \mathcal{B}_J^\mathcal{C}$ is fixed, then substituting $L_g$ by $L_f^*/\|L_f\|_0$ in Equation \eqref{cstarnorm1} yields $\|L_f\|_L \geqslant \|L_f\|_0$. Therefore, $\|L_f\| = \|L_f\|_L = \|L_f\|_0$, for all $L_f \in \mathcal{B}_J^\mathcal{C}$. But $\|\, \cdot \,\|_0$ is arbitrary, so this shows that the only C$^*$-norm on $\mathcal{B}_J^\mathcal{C}$ is obtained by restricting the operator norm $\|\, \cdot \,\|$ of $\mathcal{L}_{\tilde{\mathcal{C}}}(E_n)$. \end{proof}

\subsection*{The algebra \texorpdfstring{$\mathcal{S}_J^\mathcal{C}(\mathbb{R}^n)$}{SJC(Rn)}}$\ $

$\ $

We first prove uniqueness of the C$^*$-norm for $\mathcal{S}_J^\mathcal{C}(\mathbb{R}^n)$, for any C$^*$-algebra $\mathcal{C}$ (unital, or not). Then, we prove the spectral invariance property for $\mathcal{S}_J^\mathcal{C}(\mathbb{R}^n)$ for a unital $\mathcal{C}$.

\begin{theorem} \label{s-uniq} Let $\mathcal{C}$ be a C$^*$-algebra (unital, or not) and $J$ be a skew-symmetric linear transformation on $\mathbb{R}^n$. Then there exists only one C$^*$-norm on $\mathcal{S}_J^\mathcal{C}(\mathbb{R}^n)$. \end{theorem}

\begin{proof} Let $\|\, \cdot \,\|_0$ be a C$^*$-norm on $\mathcal{S}_J^\mathcal{C}(\mathbb{R}^n)$ and $\|\, \cdot \,\|_{\mathcal{B}_J^\mathcal{C}}$ be the (unique) C$^*$-norm of $\mathcal{B}_J^\mathcal{C}(\mathbb{R}^n)$. Our strategy will be to make good use of the corresponding result already obtained for the algebra $\mathcal{B}_J^\mathcal{C}(\mathbb{R}^n)$.

Just as in Theorem \ref{b-uniq2}, define two maps on $\mathcal{B}_J^\mathcal{C}(\mathbb{R}^n)$ by \begin{equation*} \|\, \cdot \,\|_L \colon f \longmapsto \sup \left\{\|f \times_J g\|_0: g \in \mathcal{S}_J^\mathcal{C}(\mathbb{R}^n), \|g\|_0 \leqslant 1\right\} \end{equation*} and \begin{equation*} \|\, \cdot \,\|_R \colon f \longmapsto \sup \left\{\|g \times_J f\|_0: g \in \mathcal{S}_J^\mathcal{C}(\mathbb{R}^n), \|g\|_0 \leqslant 1\right\} \end{equation*} (note that $\|\, \cdot \,\|_L$ and $\|\, \cdot \,\|_R$ are well-defined because $\mathcal{S}_J^\mathcal{C}(\mathbb{R}^n)$ is an ideal in $\mathcal{B}_J^\mathcal{C}(\mathbb{R}^n)$). Then a repetition of the arguments in Theorem \ref{b-uniq2} shows that the map $f \longmapsto \|f\|_L = \|f\|_R$ is a C$^*$-norm on $\mathcal{B}_J^\mathcal{C}(\mathbb{R}^n)$. Therefore, by the uniqueness result for C$^*$-norms on $\mathcal{B}_J^\mathcal{C}(\mathbb{R}^n)$ proved in Theorem \ref{b-uniq2} we have, in particular, $\|f\|_{\mathcal{B}_J^\mathcal{C}} = \|f\|_L = \|f\|_0$, for every $f \in \mathcal{S}_J^\mathcal{C}(\mathbb{R}^n)$. This proves that restricting $\|\, \cdot \,\|_{\mathcal{B}_J^\mathcal{C}}$ is the only way to obtain a C$^*$-norm on $\mathcal{S}_J^\mathcal{C}(\mathbb{R}^n)$. In other words, $\|\, \cdot \,\|_{\mathcal{S}_J^\mathcal{C}}$ is the only C$^*$-norm on $\mathcal{S}_J^\mathcal{C}(\mathbb{R}^n)$. \end{proof}

\begin{theorem} \label{s-specinv} Let $\mathcal{C}$ be a unital C$^*$-algebra and $J$ be a skew-symmetric linear transformation on $\mathbb{R}^n$. Then the algebra $\mathcal{S}_J^\mathcal{C}(\mathbb{R}^n)$ is spectrally invariant in its C$^*$-completion $\overline{\mathcal{S}_J^\mathcal{C}(\mathbb{R}^n)}$. \end{theorem}

\begin{proof} As a consequence of \cite[Proposition 5.2, p.~40]{rieffel}, the completion $\mathcal{A} := \overline{\mathcal{S}_J^\mathcal{C}(\mathbb{R}^n)}$ is a non-unital C$^*$-algebra, so the proof of spectral invariance of $\mathcal{B} := \mathcal{S}_J^\mathcal{C}(\mathbb{R}^n)$ in $\mathcal{A} := \overline{\mathcal{S}_J^\mathcal{C}(\mathbb{R}^n)}$ amounts to showing that $\tilde{\mathcal{B}}$ is spectrally invariant in $\tilde{\mathcal{A}}$. First note that, by the discussion in Appendix \ref{appendixa}, the unitization of $\mathcal{B}_J^\mathcal{C}(\mathbb{R}^n)$ is spectrally invariant in the unitization of $\overline{\mathcal{B}_J^\mathcal{C}(\mathbb{R}^n)}$: in fact, if $\mathcal{C}$ is unital, then $\mathcal{B}_J^\mathcal{C}(\mathbb{R}^n)$ has a unit element which coincides with that of its C$^*$-completion (see also Appendix \ref{appendixb}). If $(f, \mu) \in \tilde{\mathcal{B}}$ is invertible in $\tilde{\mathcal{A}}$, where $f \in \mathcal{S}_J^\mathcal{C}(\mathbb{R}^n)$ and $0 \neq \mu \in \mathbb{C}$, then spectral invariance of $\mathcal{B}_J^\mathcal{C}(\mathbb{R}^n)$ in its respective C$^*$-completion shows that the inverse $(f, \mu)^{-1}$ is equal to an element $g$ in the unitization of $\mathcal{B}_J^\mathcal{C}(\mathbb{R}^n)$ given by $x \longmapsto (g_0(x), \mu')$, with $g_0 \in \mathcal{B}_J^\mathcal{C}(\mathbb{R}^n)$ and $0 \neq \mu' \in \mathbb{C}$. Hence, $g_0 = - \mu^{-1} (f \times_J g_0) - \mu^{-2} f$. But because $\mathcal{S}_J^\mathcal{C}(\mathbb{R}^n)$ is an ideal in $\mathcal{B}_J^\mathcal{C}(\mathbb{R}^n)$, this shows that $g$ actually belongs to the unitization of $\mathcal{S}_J^\mathcal{C}(\mathbb{R}^n)$. This establishes the result. \end{proof}

\begin{remark} Clearly, the above two results remain valid if we substitute $\mathcal{S}_J^\mathcal{C}(\mathbb{R}^n)$ by the operator algebra $\mathcal{S}_J^\mathcal{C}$. \end{remark}

\subsection*{Other applications} $\ $

$\ $

We begin with another consequence of Theorem \ref{uniq}.

\begin{theorem} \label{smooth} Let $\mathcal{A}$ be a C$^*$-algebra (unital, or not), $G$ be a finite-dimensional Lie group with Lie algebra $\mathfrak{g}$ and $\alpha \colon g \longmapsto \alpha_g$ be a strongly continuous representation of $G$ implemented by $*$-automorphisms on $\mathcal{A}$. Then the $*$-algebra \begin{equation*} C^\infty(\alpha) := \left\{a \in \mathcal{A}: G \ni g \longmapsto \alpha_g(a) \text{ is of class } \text{C}^\infty \right\} \end{equation*} of smooth elements for the representation $\alpha$ admits only one C$^*$-norm, which is the restriction of $\|\, \cdot \,\|_\mathcal{A}$ to $C^\infty(\alpha)$. \end{theorem}

\begin{proof} Fix an ordered basis $\mathcal{B} := \left(X_k\right)_{1 \leqslant k \leqslant d}$ for $\mathfrak{g}$ and denote by $\delta_k$ the infinitesimal generator of the one-parameter group $t \longmapsto \alpha_{\exp tX_k}$ ($\exp$ denotes the exponential map of the Lie group $G$). Suppose, for the moment, that $\mathcal{A}$ is unital. Equip $C^\infty(\alpha)$ with the topology defined by the sequence $(T_k)_{k \in \mathbb{N}}$ of seminorms given by \begin{equation*}T_0(a) := \|a\|_\mathcal{A} \quad \text{and} \quad T_k(a) := \sum_{i_1, \ldots, i_k = 1}^d \frac{1}{k!} \, \|\delta_{i_1} \ldots \delta_{i_k} a\|_\mathcal{A}, \quad \text{where} \quad k \geqslant 1, \, a \in C^\infty(\alpha). \end{equation*} Then $T \colon a \longmapsto (T_k(a))_{k \in \mathbb{N}}$ is a differential norm on $C^\infty(\alpha)$ \cite[Example 6.2 (i), (ii)]{bhattdiff}, and turns it into a Fr\'echet $*$-algebra. Therefore, since $C^\infty(\alpha)$ is dense in $\mathcal{A}$, we conclude just as in Theorem \ref{b-uniq} via an application of \cite[Theorem 3.4]{bhattdiff} that $C^\infty(\alpha)$ is closed under the C$^\infty$-functional calculus of $\mathcal{A}$. But then Theorem \ref{uniq} tells us that the restriction of $\|\, \cdot \,\|_\mathcal{A}$ is the only C$^*$-norm on $C^\infty(\alpha)$.

If $\mathcal{A}$ is non-unital, then $\alpha \colon g \longmapsto \alpha_g$ extends to a strongly continuous representation $\tilde{\alpha}$ of $G$ by $*$-automorphisms on the unitization $(\tilde{\mathcal{A}}, \|\, \cdot \,\|_{\tilde{\mathcal{A}}})$, where $\tilde{\alpha}_g((a, \lambda)) := (\alpha_g(a), \lambda)$, for all $g \in G$, $a \in \mathcal{A}$ and $\lambda \in \mathbb{C}$. Since we already know that the only C$^*$-norm on $C^\infty(\tilde{\alpha})$ is the restriction of $\|\, \cdot \,\|_{\tilde{\mathcal{A}}}$, the result follows at once from a repetition of the arguments of Theorem \ref{b-uniq2}, by observing that $C^\infty(\alpha)$ is an ideal in $C^\infty(\tilde{\alpha}) = C^\infty(\alpha) \oplus \mathbb{C}$. \end{proof}

\begin{example}
In the scalar case $\mathcal{C} = \mathbb{C}$, when $E_n$ is the usual Hilbert space $L^2(\mathbb{R}^n)$, H.O. Cordes proved \cite{C} \cite[Chapter 8]{cordes} that a bounded operator $A$ on $L^2(\mathbb{R}^n)$ is a smooth vector for the canonical action of the $(2n + 1)$-dimensional Heisenberg group by conjugation if, and only if, $A = \text{Op}(a)$ for some $a \in \mathcal{B}(\mathbb{R}^{2n})$. A similar result for the $n$-dimensional torus $\mathbb{T}^n := \mathbb{R}^n/(2\pi \mathbb{Z})^n$ is also available in the scalar case \cite[Theorem 2]{cabralmelo}: if for each $y \in \mathbb{T}^n$, $T_y$ denotes the translation operator on $L^2(\mathbb{T}^n)$, then a bounded operator $A \in \mathcal{L}(L^2(\mathbb{T}^n))$ is such that the map $\mathbb{T}^n \ni y \longmapsto T_y A T_{-y}$ is smooth if, and only if, $A = \text{Op}(a_j)$ for some function $(a_j)_{j \in \mathbb{Z}^n}$ of order zero, meaning that $a_j \in C^\infty(\mathbb{T}^n)$,
\begin{equation*}
Au(x) = \frac{1}{(2 \pi)^n} \sum_{j \in \mathbb{Z}^n} a_j(x) e^{i \langle j, x \rangle} \widehat{u}_j, \quad \text{with} \quad \widehat{u}_j := \int_{\mathbb{T}^n} e^{i \langle - j, \, \cdot \, \rangle} u(\, \cdot \,),
\end{equation*}
for all $u \in C^\infty(\mathbb{T}^n)$, $x \in \mathbb{T}^n$, and that, for every multiindex $\alpha \in \mathbb{N}^n$, we have the finiteness condition $\sup \left\{|\partial^\alpha a_j(x)|; j \in \mathbb{Z}^n, x \in \mathbb{T}^n\right\} < + \infty$. Therefore, as a consequence of Theorem \ref{smooth}, the algebras $\left\{ \text{Op}(a): a \in \mathcal{B}(\mathbb{R}^{2n}) \right\}$ and $\left\{ \text{Op}(a_j): (a_j)_{j \in \mathbb{Z}^n} \text{ has order zero} \right\}$, above, admit only one C$^*$-norm.
\end{example}

Consider the $*$-algebra $\mathcal{B}_0^\mathcal{C}(\mathbb{R}^n)$, with $J = 0$; in other words, $\mathcal{B}_0^\mathcal{C}(\mathbb{R}^n)$ is just the space $\mathcal{B}^\mathcal{C}(\mathbb{R}^n)$ equipped with the usual pointwise product and involution. We now prove a corollary of Theorem \ref{b-uniq2} which relates the ``sup norm'' (see \eqref{normsb}) $\|\, \cdot \,\|_{\mathcal{B}^\mathcal{C}, 0} \colon f \longmapsto \sup_{x \in \mathbb{R}^n} \|f(x)\|_\mathcal{C}$ and the ``operator norm'' $\|f\|_{\mathcal{B}_0^\mathcal{C}} := \|L_f\|$ on $\mathcal{B}_0^\mathcal{C}(\mathbb{R}^n)$.

\begin{proposition} \label{sup-op} Let $\mathcal{C}$ be a C$^*$-algebra (unital, or not). Then the ``sup norm'' and the ``operator norm'' coincide on $\mathcal{B}_0^\mathcal{C}(\mathbb{R}^n)$. \end{proposition}

\begin{proof} The norms $\|\, \cdot \,\|_{\mathcal{B}^\mathcal{C}, 0}$ and $\|\, \cdot \,\|_{\mathcal{B}_0^\mathcal{C}}$ are both C$^*$-norms on $\mathcal{B}_0^\mathcal{C}(\mathbb{R}^n)$, so by Theorem \ref{b-uniq2} we must have $\|f\|_{\mathcal{B}^\mathcal{C}, 0} = \|f\|_{\mathcal{B}_0^\mathcal{C}}$, for all $f \in \mathcal{B}_0^\mathcal{C}(\mathbb{R}^n)$. \end{proof}

Next, we apply our results to give an alternative proof to propositions \cite[Proposition 4.11, p.~36]{rieffel}, \cite[Proposition 5.4, p.~41]{rieffel} and \cite[Proposition 5.6, p.~42]{rieffel} in a unified manner:

\begin{theorem} \label{applications} Let $\mathcal{C}$ be a C$^*$-algebra (unital, or not) and $J$ be a skew-symmetric linear transformation on $\mathbb{R}^n$. Then for every $L_f \in \mathcal{B}_J^\mathcal{C}$ we have the following properties:

\begin{enumerate}
\item $\label{ess} \|L_f\| = \sup \left\{\|L_{f \times_J g}\|: g \in \mathcal{S}^\mathcal{C}(\mathbb{R}^n), \|L_g\| \leqslant 1\right\}$.
\item If $\mathcal{C}$ is a C$^*$-subalgebra of the C$^*$-algebra $\mathcal{A}$, so that $f$ can be seen as an element of $\mathcal{B}_J^\mathcal{A}(\mathbb{R}^n)$, then $\label{subalg} \|L_f\|^\mathcal{C} = \|L_f\|^\mathcal{A}$, where $\|\, \cdot \,\|^\mathcal{C}$ and $\|\, \cdot \,\|^\mathcal{A}$ denote the corresponding operator norms.
\item If $\mathcal{A}$ is a C$^*$-algebra and $\theta \colon \mathcal{C} \longrightarrow \mathcal{A}$ is a $*$-homomorphism, then $\label{hom} \|L_{\theta f}\|^\mathcal{A} \leqslant \|L_f\|^\mathcal{C}$, where $(\theta f)(x) := \theta(f(x))$, for all $x \in \mathbb{R}^n$. If $\theta$ is injective $\mathcal{A}$, then an equality holds.
\end{enumerate} \end{theorem}

\begin{proof} The supremum on the right-hand side of equation \eqref{ess} and the map associating the number $\|L_f\|^\mathcal{A}$ to the element $L_f \in \mathcal{B}_J^\mathcal{C}$ are both C$^*$-norms on $\mathcal{B}_J^\mathcal{C}$, so \eqref{ess} and \eqref{subalg} follow from Theorem \ref{b-uniq2}.

To see that \eqref{hom} also holds, first consider the (unique) $*$-homomorphism $\tilde{\theta} \colon \tilde{\mathcal{C}} \longrightarrow \tilde{\mathcal{A}}$ between the unitizations of $\mathcal{C}$ and $\mathcal{A}$ which extends $\theta$ and sends $1_{\tilde{\mathcal{C}}}$ to $1_{\tilde{\mathcal{A}}}$. The map $L_f \longmapsto \|L_{\tilde{\theta} f}\|^{\tilde{\mathcal{A}}}$ is a C$^*$-seminorm on $\mathcal{B}_J^{\tilde{\mathcal{C}}}$, so Proposition \ref{envalg} combined with Theorem \ref{b-uniq} imply the estimate $\|L_{\tilde{\theta} f}\|^{\tilde{\mathcal{A}}} \leqslant \|L_f\|^{\tilde{\mathcal{C}}}$, for all $L_f \in \mathcal{B}_J^{\tilde{\mathcal{C}}}$. In particular, if $L_f$ belongs to the $*$-subalgebra $\mathcal{B}_J^\mathcal{C}$, then $\|L_{\theta f}\|^\mathcal{A} \leqslant \|L_f\|^\mathcal{C}$, which proves our claim. If $\theta$ is assumed to be injective, then $L_f \longmapsto \|L_{\theta f}\|^\mathcal{A}$ is actually a C$^*$-norm on $\mathcal{B}_J^\mathcal{C}$, so the desired equality follows again from Theorem \ref{b-uniq2}. \end{proof}

\begin{appendix}

\section{A remark on spectral invariance} \label{appendixa}

When viewing a unital algebra as ``not necessarily unital'', by forgetting about its unit, we face an apparent consistency problem, since two different possible definitions of spectrum seem to be available: the spectrum with respect to the algebra itself or with respect to its unitization. Fortunately, it turns out that they ``almost'' coincide. More specifically, suppose that $\mathcal{A}$ is a unital algebra; then even though there is in this case no compelling reason to do so, we can still consider its unitization $\tilde{\mathcal{A}}$, which becomes isomorphic to the direct sum of algebras $\mathcal{A} \oplus \mathbb{C}$, the isomorphism $\tilde{\mathcal{A}} \longrightarrow \mathcal{A} \oplus \mathbb{C}$ being given by $(a,\alpha) \longmapsto (\alpha 1_\mathcal{A} + a, \alpha)$. Using this fact, it is then easy to see that the two spectra of an element $a$ of $\mathcal{A}$, that in $\mathcal{A}$ and that in $\tilde{\mathcal{A}}$, are related by $\sigma_{\tilde{\mathcal{A}}}(a) = \sigma_\mathcal{A}(a) \cup \left\{0\right\}$. As a result, the spectral radius of an element $a$ of $\mathcal{A}$ is independent of which version is used: $r_{\tilde{\mathcal{A}}}(a) = r_\mathcal{A}(a)$. Moreover, $a$ is invertible in $\mathcal{A}$ if, and only if, $(a - 1_\mathcal{A}, 1)$ is invertible in $\tilde{\mathcal{A}}$, their inverses being related by $(a - 1_\mathcal{A}, 1)^{-1} = (a^{-1} - 1_\mathcal{A}, 1)$ and, similarly, $(a, \alpha)$ is invertible in $\tilde{\mathcal{A}}$ if, and only if, $\alpha \neq 0$ and $\alpha 1_\mathcal{A} + a$ is invertible in $\mathcal{A}$, their inverses being related by $(a, \alpha)^{-1} = ((\alpha 1_\mathcal{A} + a)^{-1} - \alpha^{-1} 1_\mathcal{A} , \alpha^{-1})$.

We will now show that there is no ambiguity when dealing with the concept of spectral invariance. Let $\mathcal{A}$ be an algebra and $\mathcal{B}$ be a subalgebra of $\mathcal{A}$. If $\mathcal{A}$ and $\mathcal{B}$ happen to be unital algebras such that the inclusion of $\mathcal{B}$ into $\mathcal{A}$ takes the unit of $\mathcal{B}$ to the unit of $\mathcal{A}$, both denoted by $1$, then we have in fact two potential definitions of spectral invariance and should check that they agree. Indeed, let us prove that $\mathcal{B}$ is spectrally invariant in $\mathcal{A}$ if, and only if, $\tilde{\mathcal{B}}$ is spectrally invariant in $\tilde{\mathcal{A}}$:

(a) If $(b, \beta) \in \tilde{\mathcal{B}}$ is invertible in $\tilde{\mathcal{A}}$, then $\beta \neq 0$ and $\beta 1 + b \in \mathcal{B}$ is invertible in $\mathcal{A}$, so if $\mathcal{B}$ is spectrally invariant in $\mathcal{A}$, $(\beta 1 + b)^{-1}$ belongs to $\mathcal{B}$ and hence $(b, \beta)^{-1} = ((\beta 1 + b)^{-1} - \beta^{-1} 1 , \beta^{-1})$ belongs to $\tilde{\mathcal{B}}$, proving that $\tilde{\mathcal{B}}$ is spectrally invariant in $\tilde{\mathcal{A}}$.

(b) If $b \in \mathcal{B}$ is invertible in $\mathcal{A}$, then $(b - 1, 1) \in \tilde{\mathcal{B}}$ is invertible in $\tilde{\mathcal{A}}$, so if $\tilde{\mathcal{B}}$ is spectrally invariant in $\tilde{\mathcal{A}}$, $(b - 1, 1)^{-1} = (b^{-1} - 1, 1)$ belongs to $\tilde{\mathcal{B}}$ and hence $b^{-1}$ belongs to $\mathcal{B}$, proving that $\mathcal{B}$ is spectrally invariant in $\mathcal{A}$.

There remains one other situation where some kind of ambiguity might arise, namely when $\mathcal{A}$ is unital, but its unit $1$ does not belong to $\mathcal{B}$. In this case, even if $\mathcal{B}$ has a unit of its own, we shall discard it and regard $\mathcal{B}$ as a not necessarily unital algebra, but need to understand that its unitization $\tilde{\mathcal{B}}$ now admits two different unit-preserving embeddings: one embedding mapping the unit $(0, 1)$ of $\tilde{\mathcal{B}}$ to the unit $(0, 1)$ of $\tilde{\mathcal{A}}$, and another embedding mapping the unit $(0, 1)$ of $\tilde{\mathcal{B}}$ to the unit $1$ of $\mathcal{A}$, whose image we shall denote by $\dot{\mathcal{B}}$. Note that $\dot{\mathcal{B}}$ is just the subalgebra of $\mathcal{A}$ generated by $\mathcal{B}$ and the unit $1$ of $\mathcal{A}$. We claim that $\dot{\mathcal{B}}$ is spectrally invariant in $\mathcal{A}$ if, and only if, (i) $\tilde{\mathcal{B}}$ is spectrally invariant in $\tilde{\mathcal{A}}$ and (ii) no element of $\mathcal{B}$ is invertible in $\mathcal{A}$. First of all, it is clear that spectral invariance of $\dot{\mathcal{B}}$ in $\mathcal{A}$ implies condition (ii), because if there were any element $b$ of $\mathcal{B}$ with an inverse in $\mathcal{A}$, spectral invariance would force this inverse to belong to $\dot{\mathcal{B}}$. This, in turn, would imply $1 \in \mathcal{B}$, contradicting the hypothesis that $1 \notin \mathcal{B}$. As for condition (i), suppose that $(b, \beta) \in \tilde{\mathcal{B}} \subseteq \tilde{\mathcal{A}}$ is invertible in $\tilde{\mathcal{A}}$. Then $\beta \neq 0$ and $\beta 1 + b \in \dot{\mathcal{B}}$ is invertible in $\mathcal{A}$, so if $\dot{\mathcal{B}}$ is spectrally invariant in $\mathcal{A}$, $(\beta 1 + b)^{-1}$ belongs to $\dot{\mathcal{B}}$, which means it can be written in the form $(\beta 1 + b)^{-1} = \beta' 1 + b'$ for some $\beta' \in \mathbb{C}$, $b' \in \mathcal{B}$; but multiplying this equation by $\beta 1 + b$ gives $1 = \beta \beta' 1 + \beta b' + \beta' b + bb'$, implying that $\beta' = \beta^{-1}$, and hence $(b, \beta)^{-1} = ((\beta 1 + b)^{-1} - \beta^{-1} 1, \beta^{-1}) = (b', \beta^{-1})$ belongs to $\tilde{\mathcal{B}}$. This proves that $\tilde{\mathcal{B}}$ is spectrally invariant in $\tilde{\mathcal{A}}$. For the converse, suppose that $\beta 1 + b \in \dot{\mathcal{B}}$ is invertible in $\mathcal{A}$. Then $\beta \neq 0$, due to condition (ii), and $(b, \beta) \in \tilde{\mathcal{B}} \subseteq \tilde{\mathcal{A}}$ is invertible in $\tilde{\mathcal{A}}$, so if $\tilde{\mathcal{B}}$ is spectrally invariant in $\tilde{\mathcal{A}}$, $((\beta 1 + b)^{-1} - \beta^{-1} 1, \beta^{-1}) = (b, \beta)^{-1}$ belongs to $\tilde{\mathcal{B}}$. Hence, $(\beta 1 + b)^{-1}$ belongs to $\dot{\mathcal{B}}$, proving that $\dot{\mathcal{B}}$ is spectrally invariant in $\mathcal{A}$.

\section{When are Rieffel algebras unital?} \label{appendixb}

The main goal of this section is to discuss conditions under which the algebra $\mathcal{B}_J^\mathcal{C}(\mathbb{R}^n)$ is unital. As a byproduct, we show that $\mathcal{S}_J^\mathcal{C}(\mathbb{R}^n)$ can never be unital. We will need, however, a version of the Fourier Inversion Formula for functions in $\mathcal{B}_J^\mathcal{C}(\mathbb{R}^n)$, which we quickly derive in what follows: let $f$ be a function in $\mathcal{B}^\mathcal{C}(\mathbb{R}^n)$, $0 \leqslant \phi \leqslant 1$ be a compactly supported smooth function on $\mathbb{R}^n$ which equals 1 on a neighborhood of 0 and define, for each $m \in \mathbb{N} \backslash \left\{0\right\}$ and $x \in \mathbb{R}^n$, the function $f_m(x) := \phi(x / m) f(x)$, $x \in \mathbb{R}^n$. Then for each $m \in \mathbb{N} \backslash \left\{0\right\}$ the formula \begin{equation*} \frac{1}{(2 \pi)^n} \int e^{i v \cdot (x - u)} f_m(u) \, du \, dv = f_m(x) = \int e^{2 \pi i u \cdot v} \, f_m(x + v) \, du \, dv \end{equation*} holds, and integration by parts on the right-hand side integral combined with an induction argument \cite[p.~3]{rieffel} gives \begin{equation*} \int e^{2 \pi i u \cdot v} f_m(x + v) \, du \, dv = \int e^{2 \pi i u \cdot v} \left[ \frac{1}{(1 + u^2 + v^2)^k} \, \sum_{|\alpha| \leqslant 2k} B_\alpha(u, v) \, \partial^\alpha f_m(x + v) \right] du \, dv,\end{equation*} where $k$ is an integer greater than $n/2$, each $B_\alpha$ is a bounded function and the term between brackets is just the development of $[(1 - \Delta / 4\pi^2) M_K]^k(f)$, with $M_K$ being the multiplication operator by the function $K(u, v) := (1 + u^2 + v^2)^{-k}$ and $\Delta := \sum_{j = 1}^{2n} (\partial / \partial_j)^2$. Therefore, taking the limit $m \rightarrow + \infty$ together with an application of the Dominated Convergence Theorem gives \begin{align*} f(x) &= \lim_{m \rightarrow + \infty} \int e^{2 \pi i u \cdot v} \left[ \frac{1}{(1 + u^2 + v^2)^k} \, \sum_{|\alpha| \leqslant 2k} B_\alpha(u, v) \, \partial^\alpha f_m(x + v)\right] du \, dv \\ &= \int e^{2 \pi i u \cdot v} \left[ \frac{1}{(1 + u^2 + v^2)^k} \, \sum_{|\alpha| \leqslant 2k} B_\alpha(u, v) \, \partial^\alpha f(x + v)\right] du \, dv =: \int e^{2 \pi i u \cdot v} \, f(x + v) \, du \, dv,\end{align*} for every fixed $x \in \mathbb{R}^n$, by the definition of oscillatory integrals on page 3 of the monograph \cite{rieffel}. This establishes our result (see also \cite[Corollary 1.12, p.~9]{rieffel}).

\begin{lemma} \label{b-unit} The algebra $\mathcal{B}_J^\mathcal{C}(\mathbb{R}^n)$ is unital if, and only if, the C$^*$-algebra $\mathcal{C}$ is unital.\end{lemma}

\begin{proof} If $\mathcal{C}$ is unital, then an application of the generalized Fourier Inversion Formula derived above shows that the constant function $\tilde{1} \colon x \longmapsto 1_\mathcal{C}$ satisfies $\tilde{1} \times_J f = f$, for all $f \in \mathcal{B}_J^\mathcal{C}(\mathbb{R}^n)$. Applying the involution on both sides of this equality yields $f \times_J \tilde{1} = f$, for all $f \in \mathcal{B}_J^\mathcal{C}(\mathbb{R}^n)$, so $\tilde{1}$ is the unit of $\mathcal{B}_J^\mathcal{C}(\mathbb{R}^n)$.

Conversely, suppose that $\mathcal{B}_J^\mathcal{C}(\mathbb{R}^n)$ is unital, with unit element $U \colon x \longmapsto U(x) \in \mathcal{C}$. Let us begin by showing that $U$ must be a constant function. Fix an approximate identity $(e_\alpha)_{\alpha \in \Gamma}$ for $\mathcal{C}$ and consider the constant functions $f_\alpha \colon x \longmapsto e_\alpha$, for all $\alpha \in \Gamma$. Then by the generalized Fourier Inversion Formula we obtain \begin{equation*} e_\alpha = f_\alpha(x) = (f_\alpha \times_J U)(x) = e_\alpha \int e^{2 \pi i u \cdot v} \, U(x + v) \, du \, dv = e_\alpha \, U(x), \end{equation*} for all $x \in \mathbb{R}^n$ and $\alpha \in \Gamma$. This shows that the limit $e := \lim_\alpha e_\alpha$ exists and that $e = U(x)$, for all $x \in \mathbb{R}^n$. Therefore, $U$ is the constant function $U \colon x \longmapsto e$. But if $c \in \mathcal{C}$ is fixed and $f_c$ denotes the constant function $x \longmapsto c$, we may use the generalized Fourier Inversion Formula again to obtain $c = f_c(x) = (f_c \times_J U)(x) = c \, e$ and $c = f_c(x) = (U \times_J f_c)(x) = e \, c$, for all $x \in \mathbb{R}^n$, which proves that $e$ is indeed the unit element of $\mathcal{C}$. \end{proof}

\begin{remark} We note that the proof of Lemma \ref{b-unit} shows that the algebra $\mathcal{S}_J^\mathcal{C}(\mathbb{R}^n)$ can never be unital, for any C$^*$-algebra $\mathcal{C}$. In fact, if $\mathcal{S}_J^\mathcal{C}(\mathbb{R}^n)$ were unital, with unit element $U \colon x \longmapsto U(x) \in \mathcal{C}$, then a repetition of the argument above would force $\mathcal{C}$ to be unital. Moreover, $U$ would have to be the constant function $x \longmapsto 1_\mathcal{C}$, which does not belong to $\mathcal{S}_J^\mathcal{C}(\mathbb{R}^n)$. \end{remark}

\section{A few remarks regarding non-separability} \label{appendixc}

In this section, we direct our efforts to give explicit proofs for two key lemmas found in \cite[Section 2]{melomerklen2}, in order to show that they still remain valid if we drop the requirement of separability on the C$^*$-algebra $\mathcal{C}$ (the reference \cite{melomerklen2} deals only with the Hilbert C$^*$-module $E_n$ over a separable unital C$^*$-algebra $\mathcal{C}$). The first lemma, below, contains the proof of a non-separable version of \cite[Lemma 1]{melomerklen2}. Also, we do not make the assumption that $\mathcal{C}$ is unital.

\begin{lemma} \label{toscano-merklen1} Let $\mathcal{C}$ be a C$^*$-algebra (unital, or not). For every $A \in \mathcal{L}_\mathcal{C}(E_n)$ there exists a unique operator $A \otimes I \in \mathcal{L}_\mathcal{C}(E_{2n})$ satisfying the property that $(A \otimes I)(f \otimes g) = (Af) \otimes g$, for all $f, g \in \mathcal{S}^\mathcal{C}(\mathbb{R}^n)$.\footnote{We note that, as opposed to what is done in \cite[Lemma 1]{melomerklen2}, we do not impose the hypothesis that $A \in \mathcal{L}_\mathcal{C}(E_n)$ must leave $\mathcal{S}^\mathcal{C}(\mathbb{R}^n)$ invariant; in fact, we cannot impose such a restriction since, in the definition of the map $S$, in Equation \eqref{S}, it is not clear that the operator $D \, [(\text{Ad}\,U)(-x, -\xi)(A)] \, \mathcal{F}^{-1}$ leaves $\mathcal{S}^\mathcal{C}(\mathbb{R}^n)$ invariant.} \end{lemma}

\begin{proof} First, let us fix some notations. Denote by $\lambda$ the Lebesgue measure on $\mathbb{R}^n$, and by $C_c^\infty(\mathbb{R}^n, \mathcal{C})$ the space of $\mathcal{C}$-valued compactly supported smooth functions on $\mathbb{R}^n$ (when $\mathcal{C} = \mathbb{C}$, we write simply $C_c^\infty(\mathbb{R}^n)$).

We begin by showing that the algebraic tensor product $C_c^\infty(\mathbb{R}^n, \mathcal{C}) \otimes_{\text{alg}} C_c^\infty(\mathbb{R}^n, \mathcal{C})$ is dense in $L^2(\mathbb{R}^{2n}, \mathcal{C})$ in the $L^2$-topology, via an adaptation of the proof of \cite[Lemma 1.2.31, p.~29]{analysis-bochner}. If $f$ belongs to $L^2(\mathbb{R}^{2n}, \mathcal{C})$, then $f$ can be approximated by $\lambda$-simple functions in the $L^2$-norm \cite[Lemma 1.2.19 (1), p.~23]{analysis-bochner}, so it suffices to prove that the indicator function $1_B$ of a fixed Borel-measurable subset $B$ of $\mathbb{R}^{2n}$ with finite measure can be $L^2$-approximated by an element in $C_c^\infty(\mathbb{R}^n) \otimes_{\text{alg}} C_c^\infty(\mathbb{R}^n)$. Since there exists a cube $C := \prod_{j = 1}^{2n} [c_j, d_j)$, $c_j, d_j \in \mathbb{R}$, which properly contains $B$, we may consider the (restricted) Borel $\sigma$-algebra $\mathcal{A}$ on $C$ and the subsequent algebra $\mathcal{B} \subseteq \mathcal{A}$ of finite unions of cubes of the form $\prod_{j = 1}^{2n} [a_j, b_j)$, $a_j, b_j \in \mathbb{R}$, which generates $\mathcal{A}$. But, then, given any $\epsilon > 0$, an application of \cite[Lemma A.1.2, p.~502]{analysis-bochner} shows that there exists a set $B' \in \mathcal{B}$ which satisfies $\lambda(B \Delta B') < \epsilon$, where $B \Delta B'$ is the symmetric difference $B \Delta B' := (B \cup B') \backslash (B \cap B') = (B \backslash B') \cup (B' \backslash B)$. This shows that $1_B$ can be approximated by indicator functions $1_{B'}$ in the $L^2$-norm, where $B' \in \mathcal{B}$. On the other hand, since the indicator function of an interval $[a, b)$, $a, b \in \mathbb{R}$, can be $L^2$-approximated by a function in $C_c^\infty(\mathbb{R})$, it follows that every indicator function of a cube in $\mathbb{R}^{2n}$, being a product of indicator functions of real intervals, can be $L^2$-approximated by a function in $C_c^\infty(\mathbb{R}^n) \otimes_{\text{alg}} C_c^\infty(\mathbb{R}^n)$. This proves the desired claim that $C_c^\infty(\mathbb{R}^n, \mathcal{C}) \otimes_{\text{alg}} C_c^\infty(\mathbb{R}^n, \mathcal{C})$ is dense in $L^2(\mathbb{R}^{2n}, \mathcal{C})$ implying, in particular, that $\mathcal{S}^\mathcal{C}(\mathbb{R}^n) \otimes_{\text{alg}} \mathcal{S}^\mathcal{C}(\mathbb{R}^n)$ is dense in $\mathcal{S}^\mathcal{C}(\mathbb{R}^{2n})$ in the $L^2$-norm.

Now, we treat the tensor product issue. If $\phi \colon \mathcal{C} \longrightarrow \mathcal{L}_\mathcal{C}(E_n)$ is a $*$-homomorphism, we denote by $E_n \otimes_\phi E_n$ the \textit{interior tensor product} of $E_n$ with itself \cite[p.~41]{lance}, which is a Hilbert C$^*$-module over $\mathcal{C}$: let $N$ be the vector space $N := \left\{z \in E_n \otimes_{\text{alg}} E_n: \langle z, z \rangle_\phi = 0\right\}$ \cite[Proposition 4.5, p.~40]{lance}; then the tensor product $E_n \otimes_\phi E_n$ is the Banach space completion of the quotient $(E_n \otimes_{\text{alg}} E_n) / N$ equipped with the $\mathcal{C}$-valued inner product acting on equivalence classes of simple tensors as \begin{equation*} \langle [f_1 \otimes_{\text{alg}} g_1], [f_2 \otimes_{\text{alg}} g_2] \rangle_\phi := \langle g_1, \left\{\langle f_1, f_2 \rangle_{E_n}\right\} \, g_2 \rangle_{E_n}. \end{equation*} For our purposes, we take $\phi$ as the $*$-homomorphism which sends an element $c \in \mathcal{C}$ to the left-multiplication operator $\phi(c)(f) := c \, f$, where $f \in E_n$ (note that, indeed, $\phi(c)$ belongs to $\mathcal{L}_\mathcal{C}(E_n)$). We will now show that the map \begin{equation*} \iota \colon (\mathcal{S}^\mathcal{C}(\mathbb{R}^n) \otimes_{\text{alg}} \mathcal{S}^\mathcal{C}(\mathbb{R}^n), \langle \, \cdot \,, \, \cdot \, \rangle_{E_{2n}}) \longrightarrow ((\mathcal{S}^\mathcal{C}(\mathbb{R}^n) \otimes_{\text{alg}} \mathcal{S}^\mathcal{C}(\mathbb{R}^n)) / N, \langle \, \cdot \,, \, \cdot \, \rangle_\phi), \qquad \iota(f) := [f], \end{equation*} extends to a linear isomorphism $\overline{\iota} \colon E_{2n} \longrightarrow E_n \otimes_\phi E_n$ which preserves the right $\mathcal{C}$-module structure and satisfies $\langle \overline{\iota}(z_1), \overline{\iota}(z_2) \rangle_\phi = \langle z_1, z_2 \rangle_{E_{2n}}$, for all $z_1, z_2 \in E_{2n}$. First, note that the calculation \begin{align*} \langle [f_1 \otimes_{\text{alg}} g_1], [f_2 \otimes_{\text{alg}} g_2] \rangle_\phi &= \int_{\mathbb{R}^n} g_1(s)^* \left(\int_{\mathbb{R}^n} f_1(t)^* f_2(t) \, dt\right) g_2(s) \, ds \\ &= \left(\int_{\mathbb{R}^{2n}} (f_1(t) g_1(s))^* f_2(t) \, g_2(s) \, dt \, ds \right) = \langle f_1 \otimes_{\text{alg}} g_1, f_2 \otimes_{\text{alg}} g_2 \rangle_{E_{2n}}, \end{align*} which holds for all $f_1, g_1, f_2, g_2 \in \mathcal{S}^\mathcal{C}(\mathbb{R}^n)$, shows that $\iota$ preserves the $\mathcal{C}$-valued inner product, so it is an isometry. In the previous paragraph we have proved, in particular, that $\mathcal{S}^\mathcal{C}(\mathbb{R}^n) \otimes_{\text{alg}} \mathcal{S}^\mathcal{C}(\mathbb{R}^n)$ is dense in $\mathcal{S}^\mathcal{C}(\mathbb{R}^{2n})$ with respect to the norm $\|\, \cdot \,\|_2$, so $\mathcal{S}^\mathcal{C}(\mathbb{R}^n) \otimes_{\text{alg}} \mathcal{S}^\mathcal{C}(\mathbb{R}^n)$ is $\|\, \cdot \,\|_2$-dense in $E_{2n}$. On the other hand, $\iota[\mathcal{S}^\mathcal{C}(\mathbb{R}^n) \otimes_{\text{alg}} \mathcal{S}^\mathcal{C}(\mathbb{R}^n)]$ is dense in $(E_n \otimes_{\text{alg}} E_n) / N$ with respect to the norm $\|\, \cdot \,\|_\phi$ induced by the $\mathcal{C}$-valued inner product $\langle \, \cdot \,, \, \cdot \, \rangle_\phi$, since an application of the Cauchy-Schwarz inequality for Hilbert C$^*$-modules gives \begin{align*} \|[(f - g) \otimes h]\|_\phi^2 = \|\langle (f - g) \otimes h, (f - g) \otimes h \rangle_\phi \|_\mathcal{C} &= \|\langle h, \left\{ \langle (f - g), (f - g) \rangle_{E_n} \right\} \, h \rangle_{E_n}\|_\mathcal{C} \\ &\leqslant \|f - g\|_{E_n}^2 \|h\|_{E_n}^2, \qquad f, g, h \in E_n \end{align*} (an analogous estimate holds for elements of the form $[f \otimes (g - h)]$, $f, g, h \in E_n$). Therefore, the map $\overline{\iota}$ is defined by a standard extension-by-limits argument, so the conclusion of the lemma follows from the calculation in \cite[(4.6), p.~42]{lance}: it shows that, for any given $A \in \mathcal{L}_\mathcal{C}(E_n)$, there exists a unique operator $A \otimes I \in \mathcal{L}_\mathcal{C}(E_n \otimes_\phi E_n)$ satisfying the property that $(A \otimes I)(f \otimes g) = (Af) \otimes g$, for all $f, g \in E_n$. \end{proof}

Let $\gamma_1$ and $\gamma_2$ be the (scalar-valued) functions on $\mathbb{R}$ defined by \begin{equation*} \gamma_1(t) = \begin{cases} e^{-t}, & \text{if } t \geqslant 0 \\ 0, & \text{if } t < 0\end{cases} \qquad \text{and} \qquad \gamma_2(t) = \begin{cases} t \, e^{-t}, & \text{if } t \geqslant 0 \\ 0, & \text{if } t < 0.\end{cases}\end{equation*} Then it is clear that $(1 + d/dt) \gamma_1 = \delta_t$ and $(1 + d/dt)^2 \gamma_2 = \delta_t$ \cite[Theorem 10.1, p.~351]{folland} \cite[Proposition 2.3, p.~253]{cordes}. The functions $\gamma_1$ and $\gamma_2$ will play a central role in the following lemma. It provides, in particular, a proof for \cite[Lemma 2]{melomerklen2}.

\begin{lemma} \label{toscano-merklen2} Let $\mathcal{C}$ be a C$^*$-algebra (unital, or not). For every $b \colon (x, \xi) \longmapsto b(x, \xi)$ in $\mathcal{B}_J^\mathcal{C}(\mathbb{R}^2)$, there exists a unique $a \in \mathcal{B}_J^\mathcal{C}(\mathbb{R}^2)$ such that $D(a) = b$, where $D := (1 + \partial_\xi)^2 (1 + \partial_x)^2$ is considered as an (everywhere defined) operator on $\mathcal{B}_J^\mathcal{C}(\mathbb{R}^2)$. Moreover, such $a$ is given by the formula \begin{equation} \label{D-inv} a(x, \xi) = \int_{\mathbb{R}^3} \overline{u(s, \eta)} \, e^{i s t} b(s + x, t + \xi) \, v(t, \eta) \, ds \, dt \, d\eta,\end{equation} where $u(s, \eta) := (1 + \partial_\eta)[(1 - i \eta)^2 \gamma_2(-s) \, \gamma_2(-\eta) \, e^{i s \eta}]$ and $v(t, \eta) := \gamma_1(t - \eta)/(1 + i t)^2$, for all $(x, \xi) \in \mathbb{R}^2$. \end{lemma}

\begin{proof} An application of Fubini's Theorem shows that the integrand on the right-hand side of \eqref{D-inv} is Bochner integrable on $\mathbb{R}^3$. If we define \begin{equation} \label{D-inv2} a(x, \xi) := \int_{\mathbb{R}^2} \gamma_2(s) \, \gamma_2(t) \, b(x - s, \xi - t) \, ds \, dt = e^{-x} e^{-\xi} \int_{(- \infty, x] \times (- \infty, \xi]} (x - s) \, (\xi - t) \, e^s e^t \, b(s, t) \, ds \, dt, \end{equation} a straightforward calculation shows that $D(a)(x, \xi) = b(x, \xi)$, for all $(x, \xi) \in \mathbb{R}^2$, and that $a \in \mathcal{B}_J^\mathcal{C}(\mathbb{R}^2)$. This establishes that $D$ is a surjective operator on $\mathcal{B}_J^\mathcal{C}(\mathbb{R}^2)$.

To prove the injectivity of $D$, let us first show the injectivity of $1 + \partial_x$ as an operator on $\mathcal{B}_J^\mathcal{C}(\mathbb{R}^2)$. Fix a continuous linear functional $\phi$ on $\mathcal{C}$ and suppose that $(1 + \partial_x)(f) = 0$ for some $f \in \mathcal{B}_J^\mathcal{C}(\mathbb{R}^2)$, so that $(1 + \partial_x)(\phi \circ f) = \phi[(1 + \partial_x)(f)] = 0$. Then multiplying both sides by the exponential function $x \longmapsto e^x$ and integrating from 0 to $x$ gives $(\phi \circ f)(x, \xi) = e^{-x} g(\xi)$, for a certain function $g$ defined on $\mathbb{R}$ and all $(x, \xi) \in \mathbb{R}^2$. But if $g(\xi_0) \neq 0$, for some $\xi_0 \in \mathbb{R}$, then taking the limit $x \rightarrow - \infty$ on both sides of $(\phi \circ f)(x, \xi_0) = e^{-x} g(\xi_0)$ implies that $\lim_{x \rightarrow - \infty} (\phi \circ f)(x, \xi_0) = + \infty$, contradicting the boundedness of $\phi \circ f$. Therefore, $\phi \circ f$ must be identically zero which, by a corollary of Hahn-Banach's Theorem, implies that $f$ must also be identically zero. Since the same proof applies for the operator $1 + \partial_\xi$ we have established, in particular, that $D$ is injective. Hence, $D$ is a bijective operator on $\mathcal{B}_J^\mathcal{C}(\mathbb{R}^2)$.

Finally, to prove formula \eqref{D-inv}, consider the vector space of bounded continuous $\mathcal{C}$-valued functions $f$ on $\mathbb{R}$ whose lateral derivatives exist but fail to match on at most a finite number of points of $\mathbb{R}$. Then $1 + d/ds$ sends this space into the space of all $\mathcal{C}$-valued functions on $\mathbb{R}$ in an injective way (we make the convention that $d/ds$ associates the right lateral derivative of $f$ on all of the points): indeed, if $(1 + d/ds)(f) = 0$ for such a function and $\phi$ is a continuous linear functional on $\mathcal{C}$, we can adapt the argument of the previous paragraph to conclude that, if $\left\{x_j\right\}_{1 \leqslant j \leqslant k}$ is the set of real points (ordered in an increasing manner) where the lateral derivatives of $f$ fail to match, then there exist constants $\left\{C_j\right\}_{0 \leqslant j \leqslant k}$ such that $(\phi \circ f)(x) = e^{-x} C_j$, for each $0 \leqslant j \leqslant k$ and all $x \in I_j$, where $I_0 := (- \infty, x_1]$, $I_k := [x_k, + \infty)$ and, when $k > 1$, $I_j := [x_j, x_{j + 1}]$, $1 \leqslant j \leqslant k - 1$ -- if $f$ is everywhere differentiable, then $(\phi \circ f)(x) = e^{-x} C_0$, for some constant $C_0$ and all $x \in \mathbb{R}$; but then repeating the boundedness argument of the previous paragraph yields $C_0 = 0$, and the continuity of $f$ forces $C_j = 0$, for every $1 \leqslant j \leqslant k$. Therefore, the identity \begin{equation*} f(x) = \int_\mathbb{R} \gamma_1(x - s) \, [(1 + d/ds) f](s) \, ds = \int_{(- \infty, x]} e^{x - s} [(1 + d/ds) f](s) \, ds, \qquad x \in \mathbb{R},\end{equation*} holds for all such functions $f$, as can be seen by applying the (injective) operator $1 + d/dx$ to both sides of the equality. We can use this identity to obtain \begin{align*} \int_\mathbb{R} \overline{u(s, \eta)} \, v(t, \eta) \, d\eta &= \frac{\gamma_2(-s)}{(1 + i t)^2} \int_\mathbb{R} (1 + \partial_\eta)[(1 + i \eta)^2 \, \gamma_2(-\eta) \, e^{- i s \eta}] \gamma_1(t - \eta) \, d\eta \\ &= \frac{\gamma_2(-s)}{(1 + i t)^2} \cdot (1 + i t)^2 \, \gamma_2(-t) \, e^{- i s t} = \gamma_2(-s) \, \gamma_2(-t) \, e^{- i s t}, \qquad s, t \in \mathbb{R} \end{align*} which, when substituted in Equation \eqref{D-inv2}, gives \begin{equation*} a(x, \xi) = \int_{\mathbb{R}^2} [\gamma_2(-s) \, \gamma_2(-t) \, e^{- i s t}] \, b(x + s, \xi + t) \, e^{i s t} ds \, dt = \int_{\mathbb{R}^3} \overline{u(s, \eta)} \, v(t, \eta) \, b(x + s, \xi + t) \, e^{i s t} d\eta \, ds \, dt, \end{equation*} for all $(x, \xi) \in \mathbb{R}^2$. This is exactly what we wanted. \end{proof}

\section{The relationship between \texorpdfstring{$L^2(\mathbb{R}^n, \mathcal{C})$}{L2(Rn, C)} and \texorpdfstring{$E_n$}{En}} \label{appendixd}

In this final section of the Appendix we will give a quick proof of the fact that $L^2(\mathbb{R}^n, \mathcal{C})$ is continuously embedded in $E_n$ as a dense subspace. The proof of the lemma below was taken from \cite[Proposi\c c\~ao 3.9]{merklentese}.

\begin{lemma} \label{embedding} Let $\mathcal{C}$ be a C$^*$-algebra (unital, or not). There exists a continuous injective linear map $I \colon L^2(\mathbb{R}^n, \mathcal{C}) \longrightarrow E_n$ such that $I(f) = f$, for all $f \in \mathcal{S}^\mathcal{C}(\mathbb{R}^n)$. \end{lemma}

\begin{proof} We shall denote the usual $L^2$-norm on $L^2(\mathbb{R}^n, \mathcal{C})$ by $\|\, \cdot \,\|_{L^2}$. Analogously as in Lemma \ref{toscano-merklen1}, it can be proved that $\mathcal{S}^\mathcal{C}(\mathbb{R}^n)$ is dense in $(L^2(\mathbb{R}^n, \mathcal{C}), \|\, \cdot \,\|_{L^2})$. Therefore, the identity map $i \colon (\mathcal{S}^\mathcal{C}(\mathbb{R}^n), \|\, \cdot \,\|_{L^2}) \longrightarrow (\mathcal{S}^\mathcal{C}(\mathbb{R}^n), \|\, \cdot \,\|_2)$ extends by continuity to a map $I \colon L^2(\mathbb{R}^n, \mathcal{C}) \longrightarrow E_n$ such that $\|I(g)\|_2 \leqslant \|g\|_{L^2}$, for all $g \in L^2(\mathbb{R}^n, \mathcal{C})$, and $I(f) = f$, for all $f \in \mathcal{S}^\mathcal{C}(\mathbb{R}^n)$.

We will now show that $I$ is injective. Suppose $I(f) = 0$, for a fixed $f \in L^2(\mathbb{R}^n, \mathcal{C})$, and let $(f_m)_{m \in \mathbb{N}}$ be a sequence in $\mathcal{S}^\mathcal{C}(\mathbb{R}^n)$ converging to $f$ in $(L^2(\mathbb{R}^n, \mathcal{C}), \|\, \cdot \,\|_{L^2})$. An application of H\"older's inequality shows that \begin{equation*} \left\| \int_{\mathbb{R}^n} (f - f_m)^*(x) \, g(x) \, dx \right\|_\mathcal{C} \leqslant \int_{\mathbb{R}^n} \|(f - f_m)^*(x) \, g(x)\|_\mathcal{C} \, dx \leqslant \|f - f_m\|_{L^2} \, \|g\|_{L^2},\end{equation*} for all $g \in L^2(\mathbb{R}^n, \mathcal{C})$ and $m \in \mathbb{N}$. This implies, in particular, that \begin{equation} \label{holder} \left( \langle f_m, g \rangle_{E_n} = \int_{\mathbb{R}^n} f_m(x)^* \, g(x) \, dx \right)_{m \in \mathbb{N}} \quad \text{converges to} \quad \int_{\mathbb{R}^n} f(x)^* \, g(x) \, dx \end{equation} in $\mathcal{C}$, for all $g \in \mathcal{S}^\mathcal{C}(\mathbb{R}^n)$. But continuity of $I$ implies the convergence of $(f_m)_{m \in \mathbb{N}}$ to $I(f) = 0$ in $(E_n, \|\, \cdot \,\|_2)$, so the estimate $\|\langle f_m, g \rangle_{E_n}\|_\mathcal{C} \leqslant \|f_m\|_2 \, \|g\|_2$, for all $m \in \mathbb{N}$ and $g \in \mathcal{S}^\mathcal{C}(\mathbb{R}^n)$, shows that $\lim_{m \rightarrow + \infty} \langle f_m, g \rangle_{E_n} = 0$, for each fixed $g \in \mathcal{S}^\mathcal{C}(\mathbb{R}^n)$. Combining this fact with \eqref{holder} (substituting $g$ by $f_{m'}$, $m' \in \mathbb{N}$), we obtain \begin{equation*} \int_{\mathbb{R}^n} f(x)^* \, f_{m'}(x) \, dx = 0, \qquad m' \in \mathbb{N}. \end{equation*} Then another application of H\"older's inequality gives us \begin{equation*} \int_{\mathbb{R}^n} f(x)^* \, f(x) \, dx = \lim_{m' \rightarrow + \infty} \int_{\mathbb{R}^n} f(x)^* \, f_{m'}(x) \, dx = 0, \end{equation*} from which it follows that $f = 0$. This establishes the injectivity of $I$. \end{proof}

\begin{remark} If $\mathcal{C}$ is a unital C$^*$-algebra, then the space $L^2(\mathbb{R}^n)$ is continuously embedded in $E_n$ as a subspace: in fact, the map $J \colon L^2(\mathbb{R}^n) \longrightarrow L^2(\mathbb{R}^n, \mathcal{C})$, $J(f) := f \cdot 1_\mathcal{C}$, embeds $L^2(\mathbb{R}^n)$ isometrically into $L^2(\mathbb{R}^n, \mathcal{C})$, and the composition $I \circ J$ is an isometric embedding of $L^2(\mathbb{R}^n)$ into $E_n$. \end{remark}

\end{appendix}

\end{document}